\documentclass{amsart}
\usepackage{amscd,amssymb,graphicx}

\newcommand{\co}{\colon\thinspace}

\begin{document}

\newtheorem{thm}{Theorem}
\newtheorem{ax}[thm]{Axiom}
\newtheorem{cor}[thm]{Corollary}
\newtheorem{defn}[thm]{Definition}
\newtheorem{ex}[thm]{Example}
\newtheorem{lemma}[thm]{Lemma}
\newtheorem{prop}[thm]{Proposition}
\newtheorem{remark}[thm]{Remark}

\newcommand{\sections}{\renewcommand{\thethm}{\thesection.\arabic{thm}}
           \setcounter{thm}{0}}
\newcommand{\nosections}{}
\newcommand{\subsections}{\setcounter{thm}{0}\renewcommand{\thethm}
           {\thesubsection.\arabic{thm}}}
\newcommand{\nosubsections}{\renewcommand{\thethm}{\thesection.\arabic{thm}}
           \setcounter{thm}{0}}
\newcommand{\linnum}{\stepcounter{thm}\tag{\thethm}}

\newcommand{\bC}{\mathbb{C}}
\newcommand{\bR}{\mathbb{R}}
\newcommand{\bZ}{\mathbb{Z}}
\newcommand{\cR}{\mathcal{R}}
\newcommand{\za}{\alpha}
\newcommand{\zb}{\beta}
\newcommand{\zd}{\delta}
\newcommand{\zg}{\gamma}
\newcommand{\zj}{\psi}
\newcommand{\zl}{\lambda}
\newcommand{\zm}{\mu}
\newcommand{\zn}{\nu}
\newcommand{\zp}{\pi}\newcommand{\zr}{\rho}
\newcommand{\zs}{\sigma}
\newcommand{\zt}{\tau}
\newcommand{\zx}{\xi}
\newcommand{\zv}{\varphi}
\newcommand{\zG}{\Gamma}
\newcommand{\zL}{\Lambda}

\newcommand{\SR}{S_{\mathcal{R}}}
\newcommand{\subm}{\sigma_{\mathcal{R}}}

\title{Latt\`{e}s maps and finite subdivision rules }
\author{J. W. Cannon}
\address{Department of Mathematics\\ Brigham Young University\\ Provo, UT
84602\\ U.S.A.}
\email{cannon@math.byu.edu}

\author{W. J. Floyd}
\address{Department of Mathematics\\ Virginia Tech\\
Blacksburg, VA 24061\\ U.S.A.}
\email{floyd@math.vt.edu}
\urladdr{http://www.math.vt.edu/people/floyd}

\author{W. R. Parry}
\address{Department of Mathematics\\ Eastern Michigan University\\
Ypsilanti, MI 48197\\ U.S.A.}
\email{walter.parry@emich.edu}
\date\today
\begin{abstract}
This paper is concerned with realizing Latt\`es maps as subdivision
maps of finite subdivision rules. The main result is that the Latt\`es
maps in all but finitely many analytic conjugacy classes can be
realized as subdivision maps of finite subdivision rules with one tile
type.  An example is given of a Latt\`es map which is not the subdivision
map of a finite subdivision rule with either i) two tile types
and 1-skeleton of the subdivision complex a circle or ii) one tile type.
\end{abstract}
\thanks{We thank Kevin Pilgrim for piquing our interest in realizing
Latt\`es maps as subdivision maps of finite subdivision rules.}
\maketitle

\sections

This paper is concerned with realizing rational maps by subdivision
maps of finite subdivision rules. If $\cR$ is an
orientation-preserving finite subdivision rule such that the
subdivision complex $\SR$ is a 2-sphere, then the subdivision map
$\subm$ is a postcritically finite branched map. Furthermore, $\cR$ has
bounded valence if and only if $\subm$ has no periodic critical
points. In \cite{BM} and \cite{CFP2}, Bonk-Meyer and Cannon-Floyd-Parry
prove that if $f$ is a postcritically finite rational map without periodic
critical points, then every sufficiently large iterate of $f$ is the
subdivision map of a finite subdivision rule $\cR$ with two tile types
such that the 1-skeleton of $\SR$ is a circle. Since finite
subdivision rules are essentially combinatorial objects, this gives
good combinatorial models for these iterates. It is especially
convenient to realize a postcritically finite map by the subdivision map of
a finite subdivision rule with either a single tile type or with two
tile types and 1-skeleton of the subdivision complex a circle.

While passing to an iterate is not usually a serious obstacle, it would
be preferable if one didn't need to do this. Suppose $f$ is a
postcritically finite rational map without periodic critical points.
We consider the following questions.
\begin{enumerate}
\item Is $f$ the subdivision map of a finite subdivision rule?
\item Is $f$ the subdivision map of a finite subdivision rule
with two tile types and 1-skeleton of the subdivision complex a circle?
\item Is $f$ the subdivision map of a finite subdivision rule with one
tile type?
\end{enumerate}
In this paper we consider these questions for Latt\`{e}s maps.  The
main result of this paper is to answer question 3 in the affirmative
for the maps in all but finitely many analytic conjugacy classes of
Latt\`{e}s maps.  In addition we exhibit a Latt\`{e}s map of degree 2
for which the answer to questions 2 and 3 is negative, although the
answer to question 1 is positive.

This paper has four sections.  The first section develops the setting
of Latt\`{e}s maps.  These results can be used to easily enumerate all
Latt\`{e}s maps of very small degree to test (not so easy) the above
three questions for them.

Section~\ref{sec:prune} presents the pruning lemma and immediate
consequences.  To describe the pruning lemma, let $f$ be a
Latt\`{e}$s$ map which is the subdivision map of a finite subdivision
rule with one tile type.  This obtains a subdivision complex structure
on the Riemann sphere whose 1-skeleton is a tree.  The postcritical
points of $f$ are vertices of this tree.  The pruning lemma states
that it is possible to prune this tree to obtain a subtree for which
every vertex with valence either 1 or 2 is a postcritical point and
this subtree is the 1-skeleton of a subdivision complex for which $f$
is the subdivision map.  So if $f$ is the subdivision map for a
subdivision complex whose 1-skeleton is a tree, then $f$ is the
subdivision map for a subdivision complex whose 1-skeleton is a tree
with a special form.  This special form is very restrictive.

Section~\ref{sec:largedeg} contains the main result in
Corollary~\ref{cor:main}, namely, that every map in almost every
analytic conjugacy class of Latt\`{e}s maps is the subdivision map of
a finite subdivision rule with one tile type.  This is essentially
proved in two theorems.  The first theorem,
Theorem~\ref{thm:onetilea}, states that every Latt\`{e}s map with
sufficiently large degree is the subdivision map of a finite
subdivision rule with one tile type.  The second theorem,
Theorem~\ref{thm:onetileb}, states that every nonrigid Latt\`{e}s map
is the subdivision map of a finite subdivision rule with one tile
type.  Since there are only
finitely many analytic conjugacy classes of rigid Latt\`{e}s maps with
a given degree (See the end of Section~\ref{sec:defn}.),
these two theorems give the main
result.  Theorems~\ref{thm:onetilea} and \ref{thm:onetileb} actually
say a bit more, giving additional information about the tilings.

The proof of Theorem~\ref{thm:onetilea} is much more difficult than
the proof of Theorem~\ref{thm:onetileb}.  It roughly parallels the
proof of the main theorem of \cite{CFP2}.  The strategy is to
construct a finite subdivision rule $\cR$ with one tile type, bounded
valence, and mesh approaching 0 combinatorially such that the
subdivision map $\subm$ of $\cR$ is isotopic to the Latt\`{e}s map $f$
relative to its postcritical set $P_f$.  Once this is achieved, our
results for expansion complexes \cite{CFP1} can be used to show that
$\subm$ and $f$ are in fact conjugate rel $P_f$, and so $f$ is the
subdivision map of a finite subdivision rule with one tile type.  The
finite subdivision rule $\cR$ is constructed using a lift of $f$ to
the complex plane and a standard tiling of the plane by regular
hexagons.  Our tiling of the Riemann sphere by one tile lifts to a
tiling of the plane which is combinatorially equivalent to this
standard tiling by regular hexagons.  The proof of
Theorem~\ref{thm:onetileb} does not require an isotopy and hence does
not require results for expansion complexes.  In this case our tiling
of the Riemann sphere by one tile lifts to a standard tiling of the
plane by parallelograms.

Section~\ref{sec:badquadratic} presents an example which shows that
the main result does not hold for every Latt\`{e}s map.  This example
is a quadratic Latt\`{e}s map $f$ for which both question 2 and
question 3 are false.  Question 1 is true for $f$, as we show that $f$
is the subdivision map of a finite subdivision rule with two tile
types, but the 1-skeleton of the subdivision complex is not a circle.
The Latt\`{e}s map $f$ has a lift to the complex plane of the form
$\widetilde{f}(z)=\za z$, where $\za=(1+\sqrt{-7})/2$ and the
postcritical set of $f$ lifts to the lattice generated by 1 and $\za$.
The complex conjugate of $f$ gives another such example. Yet another
example with these properties is given by a real cubic Latt\`{e}s map
$g$.  The map $g$ has a lift to the complex plane of the form
$\widetilde{g}(z)=\za z$, where $\za=\sqrt{-3}$ and the postcritical
set of $g$ lifts to the lattice generated by 1 and $(1+\sqrt{-3})/2$.
It is possible to prove these claims for $g$ much as we prove them for
$f$ in Section~\ref{sec:badquadratic}. The three analytic conjugacy
classes of Latt\`{e}s maps represented by these three maps are
probably the only ones for which question 3 is false.

As for question 1, as far as we know, every rational map with finite
postcritical set and no periodic critical points is the subdivision
map of a finite subdivision rule.

\section{Definitions and basic facts for Latt\`{e}s maps }\label{sec:defn}\nosubsections

Following Milnor \cite[Remark 3.5]{M}, we define a Latt\`{e}s map to be a
rational function from the Riemann sphere $\widehat{\bC}$ to itself
such that its local degree at every critical point is 2 and there are
exactly four postcritical points, none of which is also critical.  Let
$f\co \widehat{\bC}\to \widehat{\bC}$ be a Latt\`{e}s map.

As in Section 3.1 of \cite{M}, it follows that there exists an
analytic branched cover $\wp\co \bC\to \widehat{\bC}$ which is
branched exactly over the postcritical points of $f$ and the local
degree of $\wp$ at every branch point is 2.  (The function $\wp$ is a
Weierstrass $\wp$-function up to precomposing and postcomposing with
analytic automorphisms.) Let $\zL$ be the set of branch points of
$\wp$.  It is furthermore true that $\wp$ is a regular branched cover,
and its group of deck transformations $\zG$ is generated by the set of
all rotations of order 2 about the points of $\zL$.  We refer to $\zG$
as the \textbf{orbifold fundamental group} of $f$.  Given rotations
$z\mapsto 2 \zl-z$ and $z\mapsto 2 \zm-z$ of order 2 about the points
$\zl,\zm\in \zL$, their composition, the second followed by the first,
is the translation $z\mapsto z+2(\zl-\zm)$.  We may, and do, normalize
so that $0\in \zL$.  So $\zG$ contains a subgroup with index 2
consisting of translations of the form $z\mapsto z+2\zl$ with $\zl\in
\zL$.  It follows that $\zL$ is a lattice in $\bC$ and that the
elements of $\zG$ are the maps of the form $z\mapsto \pm z+2\zl$ for
some $\zl\in \zL$.

Douady and Hubbard show in \cite[Proposition 9.3]{DH} that
the map $f\co \widehat{\bC}\to
\widehat{\bC}$ lifts to a map $\widetilde{f}\co \bC\to \bC$ given by
$\widetilde{f}(z)=\za z+\zb$ for some imaginary quadratic algebraic
integer $\za$ (possibly an element of $\bZ$) such that $\za \zL
\subseteq \zL$ and some $\zb\in \zL$.  The following lemma is devoted
to determining to what extent $\za$, $\zb$, and $\zL$ are determined by
the analytic conjugacy class of $f$.

\begin{lemma}\label{lemma:unqss}  Let $f_0$ be a Latt\`{e}s map
which is analytically conjugate to $f$.  Let $\wp_0\co \bC\to
\widehat{\bC}$ be a branched cover for $f_0$ corresponding to $\wp$.
Let $\zL_0$ be the set of branch points of $\wp_0$, and assume that
$0\in \zL_0$.  Suppose that $f_0\co \widehat{\bC}\to \widehat{\bC}$
lifts to a map $\widetilde{f}_0\co \bC\to \bC$ given by
$\widetilde{f}_0(z)=\za_0z+\zb_0$.  Then the following hold.
\begin{enumerate}
  \item $\za_0=\pm \za$.
  \item $\zb_0=\zg \zb+\zd$ where $\zg\in \bC^\times $ and $\zd\in
(\za+1)\zL_0+2\zL_0$.
  \item $\zL_0=\zg \zL$ with $\zg$ as in statement 2.
\end{enumerate}
Conversely, if $\wp_0\co \bC\to \widehat{\bC}$ is a branched cover as
above, if $\zL_0$ is the set of branch points of $\wp_0$ with $0\in
\zL_0$, and if $\widetilde{f}_0(z)=\za_0z+\zb_0$ such that statements
1, 2, and 3 hold, then $\widetilde{f}_0$ is the lift of a Latt\`{e}$s$
map which is analytically conjugate to $f$.
\end{lemma}
  \begin{proof} Let $\zs$ be a M\"{o}bius transformation such that
$f_0=\zs\circ f\circ \zs^{-1}$.  Just as for $f$ and $f_0$, the map
$\zs\co \widehat{\bC}\to \widehat{\bC}$ lifts to a map
$\widetilde{\zs}\co \bC\to \bC$ given by $\widetilde{\zs}(z)=\zr
z+\zn$ for some $\zr\in \bC^\times $ and $\zn\in \bC$.  Since $\zs$
maps the postcritical set of $f$ to the postcritical set of $f_0$,
$\widetilde{\zs}$ maps $\zL$ to $\zL_0$, that is, $\zr
\zL+\zn=\zL_0$.  In other words, the coset $\zr  \zL+\zn$ of the group
$\zr \zL$ equals the group $\zL_0$.  The only way that a coset can be
a group is if it is the trivial coset, and so $\zn\in \zL_0$ and
$\zL_0=\zr \zL$.  Since $\widetilde{f}_0$ and $\widetilde{\zs}\circ
\widetilde{f}\circ \widetilde{\zs}^{-1}$ are both lifts of $f_0$,
they differ by an element of $\zG_0$, the group of deck
transformations of $\wp_0$.  In other words,
there exists $\zl_0\in \zL_0$ such that $\pm
\widetilde{f}_0(z)+2\zl_0=\widetilde{\zs}\circ \widetilde{f}\circ
\widetilde{\zs}^{-1}(z)$ for all $z\in \bC$.  So we have the following.
  \begin{equation*}
\begin{aligned}
\pm (\za_0z+\zb_0)+2\zl_0 & =\pm
\widetilde{f}_0(z)+2\zl_0=\widetilde{\zs}\circ \widetilde{f}\circ
\widetilde{\zs}^{-1}(z)
=\zr \widetilde{f}(\zr^{-1}(z-\zn))+\zn \\
 & =\zr(\za \zr^{-1}(z-\zn)+\zb)+\zn=\za z+\zr \zb+(1-\za)\zn
\end{aligned}
  \end{equation*}
Hence $\za_0=\pm \za$, which yields statement 1.  Furthermore
  \begin{equation*}
\pm \zb_0=\zr \zb+(1-\za)\zn-2\zl_0.
  \end{equation*}
We have seen that $\zn\in \zL_0$.  So setting $\zg=\pm \zr$ and
$\zd=\pm((\za+1)\zn-2\za \zn-2\zl_0)$, we have that $\zb_0=\zg
\zb+\zd$ with $\zg\in \bC^\times $ and $\zd\in (\za+1)\zL_0+2\zL_0$.
We now have verified statements 1, 2, and 3.

For the converse, it is a straightforward matter to construct
$\widetilde{\zs}$ such that $\widetilde{\zs}$ conjugates
$\widetilde{f}$ to $\widetilde{f}_0$ up to the action of $\zG_0$.  One
checks that $\widetilde{\zs}$ descends to a rational map $\zs\co
\widehat{\bC}\to \widehat{\bC}$ which has local degree 1 at every
point of $\widehat{\bC}$.  So $\zs$ is a M\"{o}bius transformation.
It follows that $\widetilde{f}_0$ is the lift of an analytic conjugate
of $f$.

This completes the proof of Lemma~\ref{lemma:unqss}.

\end{proof}

Lemma~\ref{lemma:unqss} implies that with an appropriate modification
of $\zb$ the lattice $\zL$ may be replaced by $\zg \zL$, where $\zg$
is any nonzero complex number, without changing the analytic conjugacy
class of $f$.  In Section 7 of Chapter 2 of the number theory book
\cite{BS} by Borevich and Shafarevich the lattices $\zL$ and $\zg \zL
$ are said to be similar.  Theorem 9 and the remark following it in
Section 7 of Chapter 2 of \cite{BS} imply that every lattice in $\bC$
is similar to a lattice with a $\bZ$-basis consisting of 1 and $\zt$,
where $\zt$ lies in the standard fundamental domain for the action of
$\text{SL}(2,\bZ)$ on the upper half complex plane.  More precisely,
$\zt$ satisfies the following inequalities.
  \begin{equation*}\linnum\label{lin:fundom}
\begin{gathered}
\text{Im}(\zt)>0 \\
-\frac{1}{2}<\text{Re}(\zt)\le \frac{1}{2}\\
|\zt|\ge 1\,\text{ and if }|\zt|= 1,\text{ then }\text{Re}(\zt)\ge 0
\end{gathered}
  \end{equation*}
Moreover there is only one such $\zt$ which satisfies these
inequalities.  This and Lemma~\ref{lemma:unqss} imply that $\zt$ is
uniquely determined by the analytic conjugacy class of $f$.

\begin{cor}\label{cor:unqss}  Let
$\widetilde{f}_0(z)=\za_0z+\zb_0$ be a linear polynomial such that
$\widetilde{f}_0(\zL)\subseteq \zL$.  Then $\widetilde{f}_0$ descends
via the branched cover $\wp\co \bC\to \widehat{\bC}$ to a rational
function which is analytically conjugate to $f$ if and only if
$\za_0=\pm \za$ and $\zb_0=\zg \zb +\zd$ where $\zg \zL=\zL$,
$\zg=e^{\pm 2 \zp i/n}$ with $n\in \{1,2,3,4,6\}$, and $\zd\in
(\za+1)\zL+2\zL$.
\end{cor}
  \begin{proof} In this situation $\zL_0=\zL$.  So just as for
$\widetilde{f}$, the containment $\zg \zL\subseteq \zL$ implies that
the complex number $\zg$ is in fact an imaginary quadratic algebraic
integer.  Because $\zg \zL=\zL$, $\zg$ is invertible, that is, it is a
unit.  But all imaginary quadratic units have the form $e^{\pm 2 \zp
i/n}$ with $n\in \{1,2,3,4,6\}$.  This discussion and
Lemma~\ref{lemma:unqss} prove Corollary~\ref{cor:unqss}.

\end{proof}

In this paragraph we consider related effects of complex conjugation.
It is easy to see that the complex conjugate of a Latt\`{e}s map is
also a Latt\`{e}s map.  By applying complex conjugation to the
Latt\`{e}s map $f$, the branched cover $\wp$, and the lift
$\widetilde{f}$ of $f$, we see that the complex conjugate of
$\widetilde{f}$ is a lift of the complex conjugate of $f$.  With
respect to finite subdivision rules, the behavior of $f$ is the same
as the behavior of $\overline{f}$, so when considering
$\widetilde{f}(z)=\za z+\zb$, we may assume that $\text{Im}(\za)\ge
0$.  Since $\widetilde{f}$ and $-\widetilde{f}$ both lift $f$, we may
also assume that $\text{Re}(\za)\ge 0$.  This shows that the
restrictions put on $\za$ in the following lemma are reasonable.

\begin{lemma}\label{lemma:mx} As above, let $\zL$ be the inverse image
in $\bC$ of the postcritical set of the Latt\`{e}s map $f\co
\widehat{\bC}\to \widehat{\bC}$, and let $\widetilde{f}(z)=\za z+\zb$
be a lift of $f$.  Suppose that 1 and $\zt$ form a $\bZ$-basis of
$\zL$.  Multiplication by $\za$ determines an endomorphism of $\zL$.
Let $\left[\begin{matrix}a & b \\ c & d \end{matrix}\right]$ be the
matrix of this endomorphism with respect to the ordered $\bZ$-basis
$(1,\zt)$.  Suppose that $\text{Im}(\za)>0$.  Then $\text{Re}(\za)\ge
0$ and $\zt$ lies in the standard fundamental domain for the action of
$\text{SL}(2,\bZ)$ on the upper half complex plane if and only if the
following inequalities are satisfied.
  \begin{equation*}
\begin{gathered}
c>0 \\
a\ge -\frac{c}{2}\\
\max\{a-c+1,-a\}\le d\le a+c\\
b\le -c\,\text{ and if }\,b=-c,\text{ then }d\ge a
\end{gathered}
  \end{equation*}

\end{lemma}
  \begin{proof} From the first column of the matrix
$\left[\begin{matrix}a & b \\ c & d \end{matrix}\right]$ it follows
that $\za=a+c \zt$.  So $\zt=\frac{\za-a}{c}$.  The matrix of the
endomorphism determined by $\za-a$ is $\left[\begin{matrix}0 & b \\ c
& d-a \end{matrix}\right]$.  Because the eigenvalues of this matrix
are $\za-a$ and $\overline{\za-a}$, its trace is twice the real part
of $\za-a$ and its determinant is the square of the modulus of
$\za-a$.  So $\text{Re}(\za-a)=\frac{d-a}{2}$ and $|\za-a|^2=-bc$.
Hence $\text{Re}(\zt)=\frac{d-a}{2c}$ and $|\zt|^2=-\frac{b}{c}$.
Similarly $\text{Re}(\za)=\frac{a+d}{2}$.

Suppose that $\text{Re}(\za)\ge 0$ and that the inequalities in
line~\ref{lin:fundom} hold.  Since $\zt=\frac{\za-a}{c}$ and
$\text{Im}(\za)>0$, the inequality $\text{Im}(\zt)>0$ implies that
$c>0$, giving the first inequality in the statement of the lemma.  For
the second inequality, we combine $\text{Re}(\za)\ge 0$ and
$\text{Re}(\zt)\le \frac{1}{2}$ to obtain $a+d\ge 0$ and $d-a\le c$,
hence $a-d\ge -c$.  So $a\ge -\frac{c}{2}$, giving the second
inequality in the statement of the lemma.  Combining
$-\frac{1}{2}<\text{Re}(\zt)\le \frac{1}{2}$ and $\text{Re}(\za)\ge 0$
obtains $-c<d-a\le c$ and $a+d\ge 0$, which easily gives the third
inequality in the statement of the lemma.  The fourth inequality
follows from the fact that $|\zt|\ge 1$ with equality only if
$\text{Re}(\zt)\ge 0$.

Proving the converse is straightforward.

This proves Lemma~\ref{lemma:mx}.

\end{proof}

\begin{lemma}\label{lemma:beta}  (1) In
Corollary~\ref{cor:unqss} the case $\zg=\pm i$ occurs only when
$\zt=i$, and the case $\zg=\pm e^{\pm 2\zp i/3}$ occurs only when
$\zt=e^{2\zp i/6}$.\\
\medskip(2) Let $\left[\begin{matrix}a & b \\ c & d
\end{matrix}\right]$ be as in Lemma~\ref{lemma:mx}, and let $M$ be the
reduction of $\left[\begin{matrix}a+1 & b \\ c & d+1
\end{matrix}\right]$ modulo 2.  Then a complete list of
distinct coset representatives of $(\za+1)\zL+2\zL$ in $\zL$ is given by
  \begin{equation*}
\begin{gathered}
0\text{ if }\text{rank}(M)=2 \\
0,1,\zt,\zt+1\text{ if }\text{rank}(M)=0\\
0,\zl\text{ if }\text{rank}(M)=1,
\end{gathered}
  \end{equation*}
where $\zl$ is any element of $\zL$ whose image in $\zL/2\zL$ is not
in the column space of $M$.
\end{lemma}
  \begin{proof} If $\zg=\pm i$, then $i\in \zL$ because $\zg
\zL\subseteq \zL$.  But then $i$ is an integral linear combination of
1 and $\zt$ with $\zt$ in the standard fundamental domain for the
action of $\text{SL}(2,\bZ)$ on the upper half complex plane.  This
implies that $\zt=i$.  A similar argument applies when $\zg=\pm e^{\pm
2\zp i/3}$.  This proves statement 1.

Statement 2 is clear.

\end{proof}

Since the elements of the orbifold fundamental group are Euclidean
isometries, $\widetilde{f}$ multiplies areas uniformly by the factor
$\deg(f)$.  Translation by $\zb$ does not change areas.
Multiplication by $\za$ multiplies lengths by $\left|\za\right|$ and
areas by $\left|\za\right|^2=\za \overline{\za}$.  Multiplication by
$\za$ also corresponds to multiplication by the matrix
$\left[\begin{matrix}a & b \\ c & d
\end{matrix}\right]$, and this multiplies areas by its determinant.
Therefore $\deg(f)=\za \overline{\za}=ad-bc$.

\begin{lemma}\label{lemma:cbound} If $a$, $b$, $c$, and $d$ satisfy the
inequalities of Lemma~\ref{lemma:mx}, then $ad-bc\ge 3c^2/4$.
\end{lemma}
  \begin{proof} The inequalities of Lemma~\ref{lemma:mx} imply that
$\left|a-d\right|\le c$.  Hence $c^2\ge (a-d)^2\ge
(a-d)^2-(a+d)^2=-4ad$.  Since $b\le -c$ and $c>0$, it follows that
$ad-bc\ge -c^2/4+c^2=3c^2/4$, as desired.

\end{proof}

Let $f$ be a Latt\`{e}s map with lift $\widetilde{f}(z)=\za z+\zb$ and
lattice $\zL$ as above.  We say that $f$ is \textbf{nonrigid} if
$\za\in \bR$, equivalently, $\za\in \bZ$.  We say that $f$ is
\textbf{rigid} if $\za\notin \bR$.  If $f$ is nonrigid, then since
multiplication by an integer stabilizes every lattice in $\bC$, $\zL$
can be arbitrary.  There are uncountably many analytic conjugacy
classes of Latt\`{e}s maps for every integer $\za\ge 2$.  On the other
hand, there are only finitely many analytic conjugacy classes of rigid
Latt\`{e}s maps with a given degree.  To see why, first note that
Lemma~\ref{lemma:cbound} and the paragraph before it imply that in the
rigid case a bound on $\deg(f)=ad-bc$ puts a bound on $c$.  We may
assume that $\text{Im}(\za)>0$.  The inequalities of
Lemma~\ref{lemma:mx} imply that a bound on $c$ puts a lower bound on
$a$ and $d$.  Because $b\le -c$, $c>0$ and $\left|a-d\right|\le c$, a
bound on $ad-bc$ puts an upper bound on both $a$ and $d$.  So a bound
on $ad-bc$ puts a bound on $c$, $a$, $d$ and therefore $b$.  So if
$\deg(f)$ is bounded, then there are only finitely many possibilities
for $a$, $b$, $c$ and $d$.  These values determine $\za$ in the upper
half plane and $\zt$.  Given $\za$ and $\zt$, there are always at most
four possibilities for $\zb$ up to equivalence.  So if $f$ is rigid
and $\deg(f)$ is bounded, then there are only finitely many
possibilities for the analytic conjugacy class of $f$.

\section{The pruning lemma}\label{sec:prune}\nosubsections

Suppose that the Latt\`{e}s map $f$ is the subdivision map of a finite
subdivision rule with one tile type.  Then the 1-skeleton of our
subdivision complex is a tree $T$ in $\widehat{\bC}$.  The object of
the next lemma, the pruning lemma, is to prune $T$ in order to
simplify our finite subdivision rule.  Since $f$ is a homeomorphism on
open cells of the subdivision complex, all postcritical points are
vertices of $T$.  The pruning lemma implies that we may assume that
all other vertices of $T$ have valence at least 3.  This lemma
generalizes to more general maps and graphs, but for simplicity we
content ourselves with the following statement.

 \begin{figure}
\centerline{\scalebox{.75}{\includegraphics{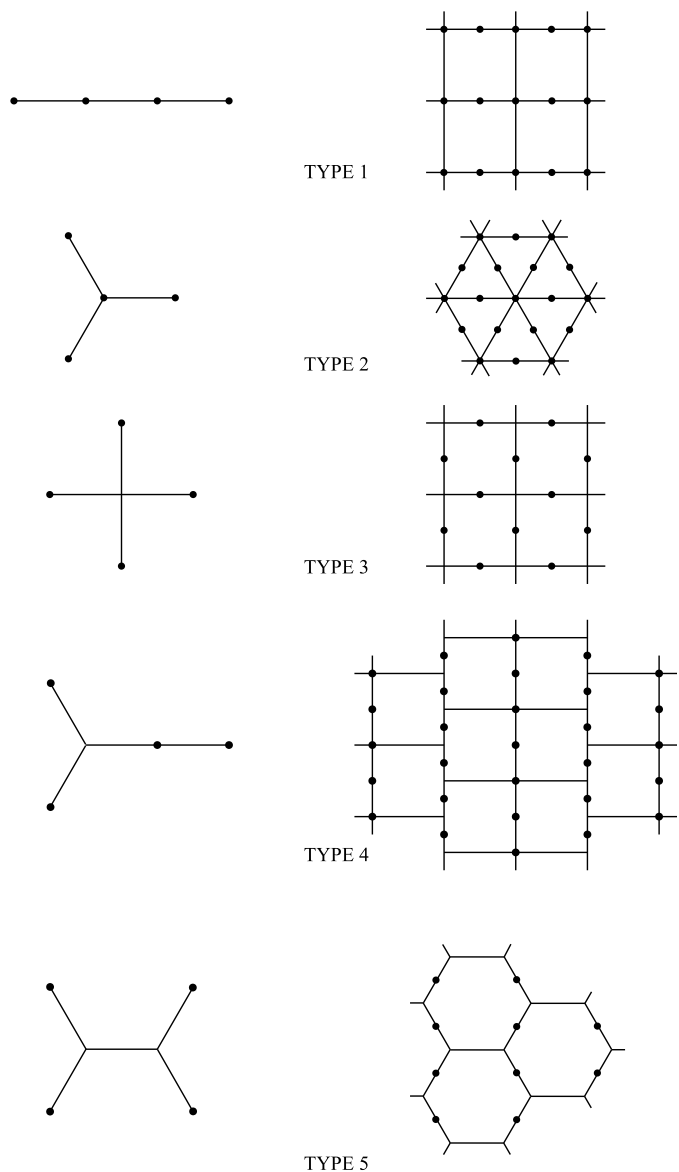}}}
\caption{ The possible trees and tilings.}
\label{fig:types}
  \end{figure}

\begin{lemma}[Pruning Lemma]\label{lemma:prune}  Suppose that the Latt\`{e}s
map $f$ is the subdivision map of a finite subdivision rule with one
tile type.  Then $f$ is the subdivision map of a finite subdivision
rule with one tile type such that every vertex in its subdivision
complex with valence either 1 or 2 is a postcritical point.
\end{lemma}
  \begin{proof} Let $T$ be the 1-skeleton of the given finite
subdivision rule.  We temporarily view $T$ just as a topological space
with no regard to vertices or edges.  Let $S$ be the subspace of $T$
which is the union of all arcs joining postcritical points.  Then $S$
is connected and hence is a topological tree.  We next show that
$f(S)\subseteq S$.  For this, let $\zg$ be an arc in $T$ joining two
postcritical points.  If the restriction of $f$ to $\zg$ is injective,
then $f(\zg)$ is an arc joining two postcritical points, and so
$f(\zg)\subseteq S$.  If the restriction of $f$ to $\zg$ is not
injective, then because $f$ maps $\zg$ into the tree $T$ there exists
a point $p\in \zg$ such that $f(\zg)$ is folded at $f(p)$.  But then
$p$ is a critical point and $f(p)$ is a postcritical point.  We see
that in general $f(\zg)$ is a union of arcs which join two
postcritical points.  So $f(S)\subseteq S$.  We make $S$ into a graph
by putting vertices at the postcritical points as well as the points
whose complements have at least three connected components.  We see
that $S$ is the 1-skeleton of a cell structure for the 2-sphere and
that $f$ is a subdivision map for this subdivision complex with one
tile type.  The vertices of $S$ with valence either 1 or 2 are
postcritical points.

This proves Lemma~\ref{lemma:prune}.

\end{proof}

Lemma~\ref{lemma:prune} shows that if the Latt\`{e}s map $f$ is the
subdivision map of a finite subdivision rule with one tile type and
1-skeleton $T$, then we may assume that every vertex of $T$ with
valence either 1 or 2 is a postcritical point.  We refer to the
vertices of $T$ which are not postcritical points as
\textbf{accidental vertices}.  The assumption that accidental vertices
have valence at least 3 severely limits the possibilities for $T$.
There are five of them.  Figure~\ref{fig:types} shows all five
possibilities for $T$.  The tile in $\widehat{\bC}$ which $T$
determines lifts to a tiling of $\bC$, and Figure~\ref{fig:types}
indicates the tiling corresponding to every possibility for $T$.  Dots
in trees indicate postcritical points, and dots in tilings indicate
inverse images of postcritical points.

We emphasize that the illustrations in Figure~\ref{fig:types} are
correct only up to isotopy.  This is true in particular for the
tilings.  The inverse image $\zL$ of the postcritical set of $f$ is a
lattice and the tiling is invariant under the orbifold fundamental
group $\zG$, which imposes a certain structure, but otherwise every
tiling is correct only up to an affine transformation followed by a
$\zG$-invariant isotopy rel $\zL$.

\section{Main Results}\label{sec:largedeg}\nosubsections

\begin{thm}\label{thm:onetilea} Every Latt\`{e}s map with sufficiently
large degree is the subdivision map of a finite subdivision rule with
one tile type, and the tile has the form of type 5 in
Figure~\ref{fig:types}.
\end{thm}
  \begin{proof} The proof begins by fixing notation and making
definitions.  Then we outline the argument.  Finally we fill in the
details.

Let $f\co \widehat{\bC}\to \widehat{\bC}$ be a Latt\`{e}s map.  Let
$\zL$ be a lift of the postcritical set of $f$, let
$\widetilde{f}(z)=\za z+\zb$ be a lift of $f$, and let $\zG$ be the
orbifold fundamental group as in Section~\ref{sec:defn}.  As in the
paragraph immediately before Lemma~\ref{lemma:mx}, we may assume that
$\text{Re}(\za)\ge 0$ and $\text{Im}(\za)\ge 0$.  As in the paragraph
containing line~\ref{lin:fundom}, we may assume that $\zL$ is a
lattice with a $\bZ$-basis consisting of 1 and $\zt$, where $\zt$
satisfies the inequalities in line~\ref{lin:fundom}; that is, we may
assume that $\zt$ lies in the standard fundamental domain for the
action of $\text{SL}(2,\bZ)$ on the upper half complex plane.

Since 1 and $\zt$ are linearly independent over $\bR$, so are
$\za^{-1}$ and $\za^{-1}\zt$.  So there exists an $\bR$-linear
isomorphism $\zj\co \bC\to \bC$ such that the points of the lattice
$\zj(\za^{-1}\zL)$ are midpoints of edges of a standard tiling $T$ of
the plane by regular hexagons and 0, $\zj(\za^{-1})$, and
$\zj(\za^{-1}\zt)$ lie in one hexagon as shown in
Figure~\ref{fig:hex}.  Let $T^*$ be the tiling of the plane by
equilateral triangles which is dual to $T$.  The vertices of $T^*$ are
the centers of the tiles of $T$.  The duality between $T$ and $T^*$
determines a bijection between the edges of $T$ and the edges of
$T^*$.  Let $S$ and $S^*$ be the tilings of the plane which are the
pullbacks of $T$ and $T^*$ under $\zj$.

  \begin{figure}
\centerline{\scalebox{.75}{\includegraphics{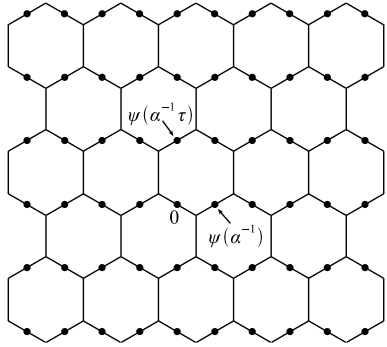}}} \caption{
The tiling $T$.}
\label{fig:hex}
  \end{figure}

In this paragraph we consider approximations to lines in the plane by
edge paths in $S$.  Let $L$ be a line in the plane, and first suppose
that $L$ does not contain a vertex of $S^*$.  We construct a subset
$\widehat{L}$ of the plane as a union of edges of $S$.  An edge $e$ of
$S$ is contained in $\widehat{L}$ if and only if $L$ meets the edge of
$S^*$ dual to $e$.  See Figure~\ref{fig:newx}, which shows part of $S$
and $S^*$.  The dots indicate points of the lattice $\za^{-1}\zL$.
Part of $L$ is shown, and part of $\widehat{L}$ is drawn with thick
line segments.  It is not difficult to see that $\widehat{L}$ is
homeomorphic to $\bR$.  If $L$ contains a vertex but not an edge of
$S^*$, then we construct $\widehat{L}$ in essentially the same way by
perturbing $L$ near every vertex of $S^*$ contained in $L$.  If $L$ is
a union of edges of $S^*$, then we perturb every edge in $L$ to a new
line segment whose endpoints are not vertices of $S^*$.  The result is
again a subset of the 1-skeleton of $S$ which is homeomorphic to
$\bR$, but in this case there are two choices for every vertex of
$S^*$ in $L$ and $\widehat{L}$ is not unique.  We call $\widehat{L}$
the \textbf{edge path approximation} to $L$.

\begin{figure}
\centerline{\scalebox{.75}{\includegraphics{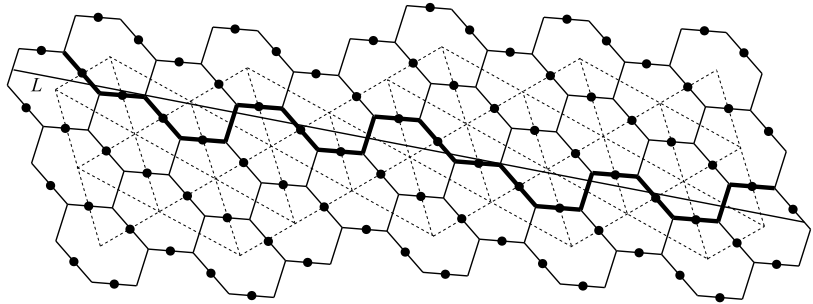}}} \caption{
Constructing the edge path approximation of the line $L$.}
\label{fig:newx}
  \end{figure}

We now outline the proof of Theorem~\ref{thm:onetilea}.  See
Figure~\ref{fig:newfundom}.  Let $L_1$ be the line containing $\zb$
and $1+\zb$.  Since $\zb\in\zL$, either $L_1$ contains the centers of
both tiles of $S$ which contain $\zb$ or $L_1$ does not contain the
center of either tile of $S$ which contains $\zb$.  If $L_1$ does not
contain the centers of the two tiles of $S$ which contain $\zb$, then
$\widehat{L}_1$ contains $\zb$.  If $L_1$ contains the centers of the
two tiles of $S$ which contain $\zb$, then $\widehat{L}_1$ is not
uniquely determined, but we may choose $\widehat{L}_1$ so that it
contains $\zb$.  So we may assume that $\widehat{L}_1$ contains $\zb$,
and likewise $1+\zb$.  Let $L_2$ be the line containing $\zt+\zb$ and
$1+\zt+\zb$.  As for $\widehat{L}_1$, we may assume that
$\widehat{L}_2$ contains both $\zt+\zb$ and $1+\zt+\zb$.  We will
later prove that $\widehat{L}_1$ and $\widehat{L}_2$ are usually
disjoint.  Let $L_3$ be the line containing $-\frac{1}{2}+\zb$ and
$-\frac{1}{2}+\zt+\zb$.  Let $P_1$ and $P_2$ be points of
$\widehat{L}_1\cap \widehat{L}_3$ and $\widehat{L}_2\cap
\widehat{L}_3$, respectively, such that the open segment of
$\widehat{L}_3$ with endpoints $P_1$ and $P_2$ is disjoint from both
$\widehat{L}_1$ and $\widehat{L}_2$.  We will later prove that $P_1$
is usually strictly between $-1+\zb$ and $\zb$ in $\widehat{L}_1$ and
that $P_2$ is usually strictly between $-1+\zt+\zb$ and $\zt+\zb$ in
$\widehat{L}_2$.  Let $Q_1$ be the image of $P_1$ under the rotation
of order 2 about $\zb$, and let $R_1$ be the image of $Q_1$ under the
rotation of order 2 about $1+\zb$.  When $\widehat{L}_1$ is uniquely
determined, the rotation of order 2 about any point of $\zL\cap L_1$
stabilizes $\widehat{L}_1$.  When $\widehat{L}_1$ is not uniquely
determined, we construct $\widehat{L}_1$ by making choices between
$\zb$ and $1+\zb$ and then extending so that the rotation of order 2
about any point of $\zL\cap L_1$ stabilizes $\widehat{L}_1$.  We
construct $\widehat{L}_2$ in the same way.  So $Q_1$ is a point of
$\widehat{L}_1$ usually strictly between $\zb$ and $1+\zb$ and $R_1$
is a point of $\widehat{L}_1$ usually strictly between $1+\zb$ and
$2+\zb$.  The composition of two rotations of order 2 is a
translation.  We translate $L_3$ by the map $z\mapsto z+R_1-P_1$ to a
line $L_4$.  The image of the closed segment of $\widehat{L}_3$ with
endpoints $P_1$ and $P_2$ under this translation is the closed segment
of $\widehat{L}_4$ with endpoints $R_1$ and $R_2$.  We will later
prove that the closed segment of $\widehat{L}_3$ with endpoints $P_1$
and $P_2$ is usually disjoint from its image under this
translation. From the segment of $\widehat{L}_1$ joining $P_1$ and
$R_1$, the segment of $\widehat{L}_4$ joining $R_1$ and $R_2$, the
segment of $\widehat{L}_2$ joining $R_2$ and $P_2$, and the segment of
$\widehat{L}_3$ joining $P_2$ and $P_1$ we obtain the hatched region
$F$, which is a union of tiles of $S$.  Let $S'$ be the tiling gotten
from $S$ simply by making the points of $\zL$ vertices.  So the
vertices of $S$ are the vertices of $S'$ with valence 3, and the
points of $\zL$ are the vertices of $S'$ with valence 2.  Whereas the
tiles of $S$ are hexagons, the tiles of $S'$ are decagons.  Let $s$ be
the tile of $S'$ containing 0, $\za^{-1}$, $\za^{-1}\zt$, and
$\za^{-1}+\za^{-1}\zt$.  There is a canonical pairing of the edges of
$s$.  The two edges of $s$ containing $0$ are interchanged by the
rotation of order 2 about $0$. The same is true at $\za^{-1}$,
$\za^{-1}\zt$, and $\za^{-1}+\za^{-1}\zt$.  The two remaining edges of
$s$ are translates of one another.  In the same way we view $F$ as
having six vertices $P_1$, $Q_1$, $R_1$, $P_2$, $Q_2$, $R_2$ of
valence 3 and four vertices $\zb$, $1+\zb$, $\zt+\zb$, $1+\zt+\zb$ of
valence 2, which we view as decomposing the boundary of $F$ into ten
edges.  The two edges of $F$ containing $\zb$ are interchanged by the
rotation of order 2 about $\zb$.  The same is true at $1+\zb$,
$\zt+\zb$, and $1+\zt+\zb$.  The two remaining edges of $F$ are
translates of one another.  It follows that $F$ is a fundamental
domain for the orbifold fundamental group $\zG$.  The image of every
tile of $S'$ under $\widetilde{f}$ is isotopic rel $\zL$ to the image
of $F$ under some element of $\zG$, and these isotopies can be made
$\zG$-equivariant.  Because $f$ respects the edge pairings of $s$ and
$F$ up to isotopy, when we descend to $\widehat{\bC}$, the tiling of
$F$ by the tiles of $S'$ determines a finite subdivision rule $\cR$
with one tile type.  The subdivision rule $\cR$ has bounded valence,
in fact valences of vertices are bounded by 3.  We will prove that the
mesh of $\cR$ usually approaches 0 combinatorially.  The result is a
finite subdivision rule $\cR$ with bounded valence and mesh
approaching 0 combinatorially such that the subdivision map $\subm$ of
$\cR$ is isotopic to $f$ rel $P_f$.  From here we proceed as in the
proof of the main theorem of \cite{CFP2} to conclude that $\subm$ and
$f$ are in fact conjugate rel $P_f$, and so $f$ is the subdivision map
of a finite subdivision rule with one tile type.

  \begin{figure}
\centerline{\scalebox{.75}{\includegraphics{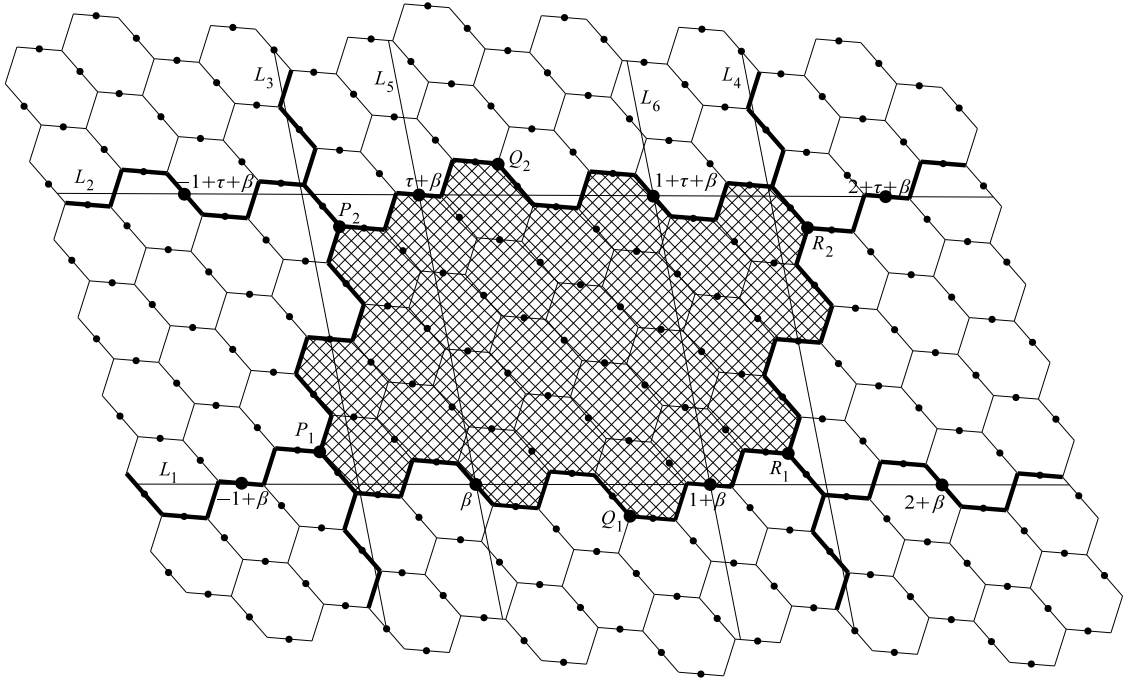}}}
\caption{ Constructing a fundamental domain when
$\za=\frac{9+\sqrt{-23}}{2}$ and $\zt=\frac{-1+\sqrt{-23}}{4}$.}
\label{fig:newfundom}
  \end{figure}

The previous paragraph reduces the proof of Theorem~\ref{thm:onetilea}
to the construction of an appropriate fundamental domain $F$ for the
orbifold fundamental group $\zG$.  This in turn reduces to verifying
certain claims made in the previous paragraph.  The word ``usually''
occurs a few times.  ``Usually'' means that the degree of $f$ is
sufficiently large and that $\za\notin \{\zt,1+\zt,2+\zt\}$.  The
claims in the previous paragraph do not always hold if $\za\in
\{\zt,1+\zt,2+\zt\}$ and sometimes when they do, the proofs below
do not.  The cases in which $\za\in \{\zt,1+\zt,2+\zt\}$ will be
handled separately at the end of the proof.  When edge path
approximations are not unique, the claims mean that it is possible to
choose edge path approximations so that the claims are true.

In this paragraph we make an observation concerning a type of
symmetry.  The outline in the next-to-last paragraph constructs a
fundamental domain $F$ for the Latt\`{e}s map with lift $z\mapsto \za
z+\zb$ and lattice generated by 1 and $\zt$.  The rotation of order 2 about
$1+\zt+\zb$ takes $F$ to the corresponding fundamental domain for the
Latt\`{e}s map with lift $z\mapsto \za z+1+\zt+\zb$ and lattice
generated by 1 and $\zt$.  So for example, by this symmetry if for a
fixed $\za$ and $\zt$ and for every $\zb\in \zL$ no tile of $S'$ meets an
edge of $F$ containing $\zb$ and an edge of $F$ containing $\zt+\zb$,
then for the same fixed $\za$ and $\zt$ and for every $\zb\in\zL$
no tile of $S'$ meets an edge of $F$ containing $1+\zb$ and an edge of $F$
containing $1+\zt+\zb$.  We will use this symmetry to simplify arguments.

We prepare to verify the claims in the outline by turning our
attention to edge path approximations.

\begin{lemma}\label{lemma:approximations} Let $L$ and $L'$ be parallel
lines in the plane.
\begin{enumerate}
  \item If no edge of $S^*$ meets both $L$ and $L'$, then
$\widehat{L}$ and $\widehat{L}'$ are disjoint.
  \item If there does not exist a line segment which is the union of
two adjacent edges of $S^*$ such that this line segment meets both $L$
and $L'$, then no tile of $S'$ meets both $\widehat{L}$ and
$\widehat{L}'$.
\end{enumerate}
\end{lemma}
  \begin{proof} Note that because the vertices of $S$ have valence 3,
if $\widehat{L}$ and $\widehat{L}'$ are not disjoint, then they have
an edge of $S$ in common.  With this observation statement 1 follows
from the definition of edge path approximation.  To prove statement 2,
let $L''$ be the line parallel to both $L$ and $L'$ such that $L''$ is
equidistant to $L$ and $L'$.  Suppose that there does not exist a
line segment which is the union of two adjacent edges of $S^*$ such
that this line segment meets both $L$ and $L'$.  Then no edge of $S^*$
meets both $L$ and $L''$.  Likewise no edge of $S^*$ meets both $L'$
and $L''$.  So statement 1 implies that $\widehat{L}$, $\widehat{L}'$,
and $\widehat{L}''$ are mutually disjoint. Moreover $\widehat{L}''$ is
between $\widehat{L}$ and $\widehat{L}'$.  It follows that no tile of
$S'$ meets both $\widehat{L}$ and $\widehat{L}'$.

This proves Lemma~\ref{lemma:approximations}.

\end{proof}

\begin{lemma}\label{lemma:onetwo} If the degree of $f$ is sufficiently
large, then no tile of $S'$ meets both $\widehat{L}_1$ and
$\widehat{L}_2$.
\end{lemma}
  \begin{proof} We prepare to apply statement 2 of
Lemma~\ref{lemma:approximations}.  We consider whether or not there
exists a line segment which is the union of two adjacent edges of
$S^*$ such that this line segment meets both $L_1$ and $L_2$.  First
suppose that these two edges of $S^*$ are translates of the line
segment joining 0 and $2\za^{-1}$.  Since $\zt+\zb \in L_2$ and $L_1$
is the translate of the real axis by $\zb$, for such edges we are led
to consider real numbers $r$ and $s$ such that $\zt+\zb+2r
\za^{-1}=s+\zb$, that is, $\zt+2r \za^{-1}=s$.  If $\za\in \bR$, then
there are no such numbers $r$ and $s$ and there is no such line
segment which meets both $L_1$ and $L_2$.  Otherwise $r$ and $s$
exist, and if $\left|r\right|>2$, then there does not exist such a
line segment which meets both $L_1$ and $L_2$.  The equation $\zt+2r
\za^{-1}=s$ implies that $\za=r\frac{2}{s-\zt}$.  Since
$\text{Im}(\zt)$ is bounded from 0, the complex number $s-\zt$ is
bounded from 0 uniformly in $\zt$ as $s$ varies over $\bR$.  So
$r\frac{2}{s-\zt}$ is bounded uniformly in $\zt$ for
$\left|r\right|\le 2$ and $s\in \bR$.  Hence $\za$ is bounded for
$\left|r\right|\le 2$ and $s\in \bR$.  But since $\text{deg}(f)=\za
\overline{\za}$, it follows that there is no such line segment meeting
both $L_1$ and $L_2$ if the degree of $f$ is sufficiently large.

In addition to having edges in the direction of $2\za^{-1}$, the
tiling $S^*$ has edges in the direction of $\za^{-1}(1+\zt)$ and
$\za^{-1}(-1+\zt)$.  We next perform analogous verifications for these
edges.  Once these verifications are complete, we may conclude that no
tile of $S'$ meets both $\widehat{L}_1$ and $\widehat{L}_2$ if the
degree of $f$ is sufficiently large.

Now we consider edges of $S^*$ in the direction of $\za^{-1}(1+\zt)$.
In this case we obtain the equation $\zt+r \za^{-1}(1+\zt)=s$.  So
$\za=r\frac{1+\zt}{s-\zt}$.  For all real numbers $s$ we have that
  \begin{equation*}\linnum\label{lin:argtau}
\left|s-\zt\right|\ge
\text{Im}(\zt)=\left|\zt\right|\sin(\arg(\zt))\ge
\frac{\sqrt{3}}{2}\left|\zt\right|,
  \end{equation*}
and so
  \begin{equation*}
\left|\frac{1+\zt}{s-\zt}\right|\le
\frac{2}{\sqrt{3}}\frac{1+\left|\zt\right|}{\left|\zt\right|}\le
\frac{4}{\sqrt{3}}.
  \end{equation*}
So $r\frac{1+\zt}{s-\zt}$ is bounded for $\left|r\right|\le 2$ and
$s\in \bR$.  Again it follows that there is no such line segment
meeting both $L_1$ and $L_2$ if the degree of $f$ is sufficiently
large.

For edges of $S^*$ in the direction of $\za^{-1}(-1+\zt)$ we replace
$1+\zt$ by $-1+\zt$ in the previous paragraph and again conclude that
there is no such line segment meeting both $L_1$ and $L_2$ if the
degree of $f$ is sufficiently large.  Thus if the degree of $f$ is
sufficiently large, then no tile of $S'$ meets both $\widehat{L}_1$
and $\widehat{L}_2$.

This proves Lemma~\ref{lemma:onetwo}.

\end{proof}

As in Figure~\ref{fig:newfundom}, let $L_5$ be the line containing
$\zb$ and $\zt+\zb$ and let $L_6$ be the line containing $1+\zb$ and
$1+\zt+\zb$.  It would be convenient to have a lemma for $L_5$ and
$L_6$ analogous to Lemma~\ref{lemma:onetwo} for $L_1$ and $L_2$.
Unfortunately, the situation for $L_5$ and $L_6$ is more complicated
than for $L_1$ and $L_2$.  This takes us to the following two lemmas.

\begin{lemma}\label{lemma:threesixone} Let $L$ and $L'$ be lines in
the plane parallel to $L_5$ such that the distance between $L$ and
$L'$ is at least half the distance between $L_5$ and $L_6$.  If the
degree of $f$ is sufficiently large, then no line segment which is the
union of two edges of $S^*$ in the direction of $\za^{-1}$ meets both
$L$ and $L'$.
\end{lemma}
  \begin{proof} We consider whether or not there exists a line segment
meeting both $L$ and $L'$ which is the union of two edges of $S^*$
which are translates of the line segment joining 0 and $2 \za^{-1}$.
As in the proof of Lemma~\ref{lemma:onetwo}, we obtain the equation $x
+ 2 r \za^{-1} = s\zt$, and we wish to have no solution with $x \ge
1/2$, $|r| \le 2$, and $s\in \bR$.  So we wish to have no solution to
the equation $1 + 2 r \za^{-1} = s\zt$ with $|r| \le 4$ and $s\in
\bR$.  Solving the last equation for $\za$ shows that $\za=r\frac{2}{s
\zt-1}$.  The argument in line~\ref{lin:argtau} with $\zt^{-1}$
instead of $\zt$ shows for all real numbers $s$ that
$\left|s-\zt^{-1}\right|\ge \frac{\sqrt{3}}{2}\left|\zt\right|^{-1}$,
and so $\left|s \zt-1\right|\ge \frac{\sqrt{3}}{2}$.  So
$r\frac{2}{s\zt-1}$ is bounded uniformly in $\zt$ for
$\left|r\right|\le 4$ and $s\in \bR$.  Hence $\za$ is bounded for
$\left|r\right|\le 4$ and $s\in \bR$.  Since $\text{deg}(f)=\za
\overline{\za}$, this proves that there is no such line segment if the
degree of $f$ is sufficiently large.

This proves Lemma~\ref{lemma:threesixone}.

\end{proof}

\begin{lemma}\label{lemma:threesixtwo}Let $L$ and $L'$ be lines in the
plane parallel to $L_5$.  Let $a$ and $c$ be the integers such that
$\za=a+c \zt$.
\begin{enumerate}
  \item If the distance between $L$ and $L'$ is at least half the
distance between $L_5$ and $L_6$ and if either $c\ge 11$ or $c=0$ with
$a\ge 5$, then no line segment which is the union of two edges of
$S^*$ in the direction of either $\za^{-1}(1+\zt)$ or
$\za^{-1}(-1+\zt)$ meets both $L$ and $L'$.
  \item If the distance between $L$ and $L'$ is at least twice
the distance between $L_5$ and $L_6$, then usually no edge
of $S^*$ in the direction of either $\za^{-1}(1+\zt)$ or
$\za^{-1}(-1+\zt)$ meets both $L$ and $L'$.
\end{enumerate}
\end{lemma}
  \begin{proof} We prove both statements together.  We first consider
edges of $S^*$ in the direction of $\za^{-1}(1+\zt)$.  Arguing as in
the proof of Lemma~\ref{lemma:threesixone}, we obtain the equation
$1+r \za^{-1}(1+\zt)=s \zt$.  For statement 1 we want there to be no
solution in real numbers $r$ and $s$ with $\left|r\right|\le 4$.  For
statement 2 the restriction on $r$ is $\left|r\right|\le\frac{1}{2}$.
Solving for $\za$ shows that $\za=-r\frac{1+\zt}{1-s \zt}$.

We justify every step of the next display immediately after it.
  \begin{equation*}
c\text{Im}(\zt)=
\text{Im}(\za)\le \left|\za\right|=
\left|r\right|\left|\frac{1+\zt}{1-s \zt}\right|\le
\left|r\right|\frac{4\left|\zt\right|}{\sqrt{3}}\le
\left|r\right|\frac{8}{3}\text{Im}(\zt)
  \end{equation*}
The two equations and the first inequality are straightforward.  For
the second inequality, we use the fact that $1\le \left|\zt\right|$ to
obtain $\left|1+\zt\right|\le 1+\left|\zt\right|\le
2\left|\zt\right|$.  Using line~\ref{lin:argtau} as in the proof of
Lemma~\ref{lemma:threesixone}, we obtain $\left|1-s \zt\right|\ge
\frac{\sqrt{3}}{2}$.  Combining the results of the last two sentences
gives the second inequality of the display.  Line~\ref{lin:argtau}
implies that $\left|\zt\right|\le \frac{2}{\sqrt{3}}\text{Im}(\zt)$.
This gives the last inequality of the above display.  From the display
we conclude that $c\le \left|r\right|\frac{8}{3}$.  Hence if
$\left|r\right|\le 4$, then $c\le \frac{32}{3}$ and if
$\left|r\right|\le \frac{1}{2}$, then $c\le \frac{4}{3}$.  This proves
statement 1 for edges of $S^*$ in the direction of $\za^{-1}(1+\zt)$
except when $c=0$ and it proves statement 2 for edges of $S^*$ in the
direction of $\za^{-1}(-1+\zt)$ except when $c\in \{0,1\}$.

To continue, we consider the image of the extended real line under the
linear fractional transformation $z\mapsto \frac{1+\zt}{1-z \zt}$.
One verifies that $\infty\mapsto 0$, $-1\mapsto 1$, and $0\mapsto
1+\zt$.  Since linear fractional transformations map circles and
extended lines to circles and extended lines, it follows that the
image of the extended real line is the circle $C$ containing 0, 1, and
$1+\zt$.

Thus if $\left|r\right|\le 4$ and $c=0$, then $0<\za\le 4$.  This
completes the proof of statement 1 for edges of $S^*$ in the direction
of $\za^{-1}(1+\zt)$.  If $\left|r\right|\le \frac{1}{2}$ and $c=0$,
then $0<\za\le \frac{1}{2}$.  Hence statement 2 is true for edges of
$S^*$ in the direction of $\za^{-1}(1+\zt)$ if $c=0$.

It remains to prove statement 2 for edges of $S^*$ in the direction of
$\za^{-1}(1+\zt)$ when $c=1$.  Recall that $C$ contains 0, 1, and
$1+\zt$.  This implies that the center of $C$ is on the line given by
$\text{Re}(z)=\frac{1}{2}$. Every point of the circle $C'$ through 0
and 1 with center $\frac{1}{2}$ has imaginary part at most
$\frac{1}{2}$.  Since $\text{Im}(1+\zt)\ge
\frac{\sqrt{3}}{2}>\frac{1}{2}$, the imaginary part of the center of
$C$ is positive.  Comparing $C$ with $C'$, we see that no points of
the form $a-\zt$ with $a\in \mathbb{R}$ are within $C$.  Since
$\text{Re}(1+\zt)>\frac{1}{2}$, the complex number $1+\zt$ is in
the right half of $C$.  So no complex number of the form $a+\zt$ with
$a>1$ is within $C$.  Hence no such number is within $\frac{1}{2}C$.
Since ``usually'' means that $a\ge 3$ if $c=1$, this completes the
proof of statement 2 for edges of $S^*$ in the direction of
$\za^{-1}(1+\zt)$.

For edges of $S^*$ in the direction of $\za^{-1}(-1+\zt)$, we argue in
the same way.  The only modification occurs at the very end.  Whereas
before the circle $C$ contains $1+\zt$, now it contains $-1+\zt$.
This is insufficient to conclude that no complex number of the form
$a+\zt$ with $a>1$ is within $C$.  For this we verify that the linear
fractional transformation $z\mapsto \frac{-1+\zt}{1-z \zt}$ maps the
real number $\frac{-1+\zt+\overline{\zt}}{\zt \overline{\zt}}$ to
$-\overline{\zt}=-2\text{Re}(\zt)+\zt$.  Now we can conclude that no
complex number of the form $a+\zt$ with $a>1$ is within $C$.

This proves Lemma~\ref{lemma:threesixtwo}.
\end{proof}

We now consider the claims made in the outline of
Theorem~\ref{thm:onetilea}.  The first claim is that $\widehat{L}_1$
and $\widehat{L}_2$ are usually disjoint.  This follows from
Lemma~\ref{lemma:onetwo}.  The next claim is that $P_1$ is usually
strictly between $-1+\zb$ and $\zb$ in $\widehat{L}_1$.  This is the
main content of the first statement of the next lemma.  The other two
statements are closely related results which will be used later.

\begin{lemma}\label{lemma:pone} \emph{(1)} The point $P_1$ is usually
strictly between $-1+\zb$ and $\zb$ in $\widehat{L}_1$ and $P_1$
is usually not adjacent to either $-1+\zb$ or $\zb$ in the tiling $S'$.
\begin{enumerate}
  \item [(2)] Usually no tile of $S'$ meets both $\widehat{L}_3$
between $P_1$ and $P_2$ and $\widehat{L}_1$ between $Q_1$ and $R_1$.
  \item [(3)] Usually no tile of $S'$ meets both $\widehat{L}_4$
between $R_1$ and $R_2$ and $\widehat{L}_1$ between $P_1$ and $Q_1$.
\end{enumerate}
\end{lemma}
  \begin{proof} We prove all three statements together.  As usual, let
$a$ and $c$ be the integers such that $\za=a+c \zt$.  We begin by
proving Lemma~\ref{lemma:pone} for either $c=0$ or $c\ge 11$.  In this
case statement 2 of Lemma~\ref{lemma:approximations},
Lemma~\ref{lemma:threesixone}, and statement 1 of
Lemma~\ref{lemma:threesixtwo} combine to imply that usually no tile of
$S'$ meets both $\widehat{L}_3$ and $\widehat{L}_5$.  (The condition in
Lemma~\ref{lemma:threesixtwo} that $a\ge 5$ can be met by taking the
degree of $f$ to be sufficiently large.)  Since $\zb\in L_5$, we may
assume that $\zb\in \widehat{L}_5$.  So we have that $\zb\in
\widehat{L}_1\cap \widehat{L}_5$, and it is easy to see that the
points of $\widehat{L}_1$ which are on the same side of $\zb$ as
$1+\zb$ either lie on $\widehat{L}_5$ or are on the same side of
$\widehat{L}_5$ as $1+\zb$.
Since $\widehat{L}_3$
is usually strictly on the other side of $\widehat{L}_5$ and $P_1\in
\widehat{L}_3$, it follows that $P_1$ is usually strictly on the same
side of $\zb$ as $-1+\zb$ in $\widehat{L}_1$ and that $P_1$ is
usually not adjacent to $\zb$ in $S'$.
This argument also shows that usually no
tile of $S'$ meets $\widehat{L}_3$ between $P_1$ and $P_2$ and
$\widehat{L}_1$ between $Q_1$ and $R_1$.
We have so far proved that $P_1$ is usually strictly on the same
side of $\zb$ as $-1+\zb$ in $\widehat{L}_1$, that $P_1$ is usually not
adjacent to $\zb$ in $S'$, and
that statement 2 is true if either $c=0$ or $c\ge 11$.
Analogous arguments prove that $P_1$ is usually strictly
on the same side of $-1+\zb$ as $\zb$ in $\widehat{L}_1$, that $P_1$
is usually not adjacent to $-1+\zb$ in $S'$, and that statement 3
is true if either $c=0$ or $c\ge 11$.  This proves
Lemma~\ref{lemma:pone} if either $c=0$ or $c\ge 11$.

So suppose that $1\le c\le 10$ for the rest of the proof of
Lemma~\ref{lemma:pone}.  We partition this case into two subcases
according to whether $a<3c$ or $a\ge 3c$.  Let $b$ and $d$ be the
integers such that $\za \zt=b+d \zt$.  If $a<3c$, then both $a$ and
$c$ are bounded.  With $a$ and $c$ bounded, Lemma~\ref{lemma:mx}
implies that $d$ is bounded.  Because $\deg(f)=ad-bc$ and this degree
may be taken to be arbitrarily large, the integer $b$ may be taken to
be arbitrarily negative.  So if $a<3c$, then $d$ is bounded and $b$
may be taken to be arbitrarily negative. On the other hand, if $a\ge
3c$, then Lemma~\ref{lemma:mx} implies that $d>0$.  Of course, $b<0$.

In this paragraph suppose that $a\ge 3c$.  We will prove statement 2
and half of statement 1 under this assumption.  Let $X=\{x
\za^{-1}+y\za^{-1}\zt+\zb:x,y\in \bR,x\le -\frac{3}{2}\}$, a closed
half plane.  Let $\overline{X}$ be the union of the tiles of $S'$
contained in $X$.  Both $\zj(\overline{X})$ and the image of
$\zj(\overline{X})$ under the rotation of order 2 about $\zj(\zb)$ are
shaded in Figure~\ref{fig:atleast3}.  There are essentially two
possibilities, hence two parts to Figure~\ref{fig:atleast3}, depending
on whether the edge of $T$ containing $\zj(\zb)$ has negative or
positive slope.  It is possible to choose $\widehat{L}_1$ so that
the portion of $\zj(\widehat{L}_1)$ in the region shown
in Figure~\ref{fig:atleast3} is contained in the hatched region which is
pinched near $\zj(\zb)$.  The lines $L_1$
and $L_3$ meet at
$-\frac{1}{2}+\zb=-\frac{a}{2}\za^{-1}-\frac{c}{2}\za^{-1}\zt+\zb$,
and this point is in $X$ since $a\ge 3c\ge 3$.  Moreover, this point
is in the boundary of $X$ if and only if $a=3$ and $c=1$.  Since
$b<0$, the half of $L_3$ with endpoint $-\frac{1}{2}+\zb$ which
contains $-\frac{1}{2}+\zt+\zb=-\frac{1}{2}+b \za^{-1}+d
\za^{-1}\zt+\zb$ is in $X$.  So the half of $\widehat{L}_3$ with
endpoint $P_1$ which contains $P_2$ is contained in $\overline{X}$.
Similarly, the half of $\widehat{L}_1$ with endpoint $P_1$ which does
not contain $\zb$ is contained in $\overline{X}$.  So the half of
$\widehat{L}_1$ with endpoint $Q_1$ which does not contain $\zb$ is
contained in the image of $\overline{X}$ under the rotation of order 2
about $\zb$.  We see that
$P_1$ is strictly on the same side of $\zb$ as $-1+\zb$ in
$\widehat{L}_1$ and that $P_1$ is not adjacent to $\zb$ in $S'$.
Moreover no tile of $S'$ meets both
$\widehat{L}_3$ between $P_1$ and $P_2$ and $\widehat{L}_1$ between
$Q_1$ and $R_1$.  We have just proved statement 2 and half of
statement 1 when $1\le c\le 10$ and $a\ge 3c$.

  \begin{figure}
\centerline{\includegraphics{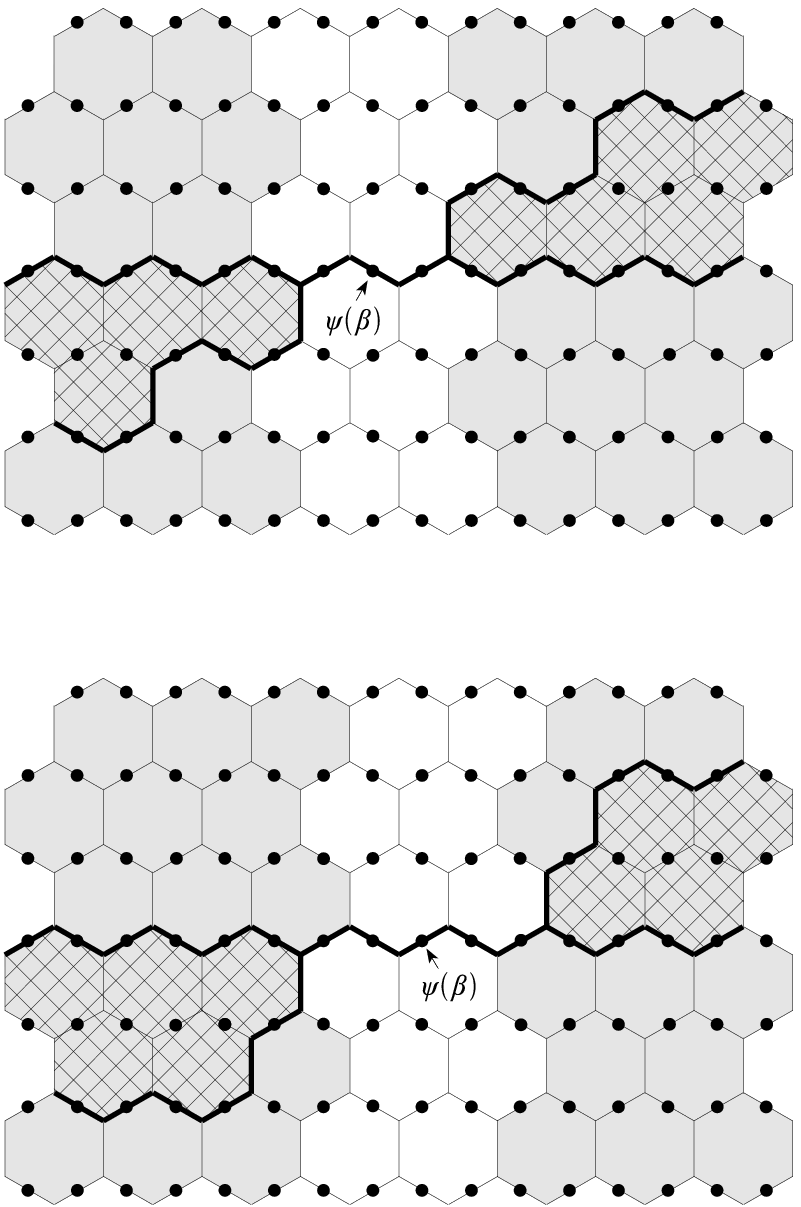}} \caption{Proving Lemma~\ref{lemma:pone}.}
\label{fig:atleast3}
  \end{figure}

Next suppose that $a<3c$.  We will prove statement 2 and the same half
of statement 1 under this assumption.  Recall from the definition of
usually that $\za\notin \{\zt,1+\zt,2+\zt\}$.  This together with the
inequality $a<3c$ implies that we may assume that $c\ne 1$.  For every
integer $m$ we consider the closed half plane $Y_m=\{x \za^{-1}+y
\za^{-1}\zt+\zb:y\ge m-\frac{1}{2}\}$.  Let $\overline{Y}_m$ be the
union of the tiles of $S'$ contained in $Y_m$.  Now let $m$ be the
ceiling of $-\frac{c}{2}$.  Again $L_1$ and $L_3$ meet at
$-\frac{1}{2}+\zb=-\frac{a}{2}\za^{-1}-\frac{c}{2}\za^{-1}\zt+\zb$.
This point is in $Y_m$.  Because $L_3$ contains
$-\frac{1}{2}+\zb$ and $-\frac{1}{2}+\zt+\zb=
-\frac{1}{2} + b \za^{-1}+d
\za^{-1}\zt+\zb$ and because $d$ is bounded and $b$ can be made
arbitrarily negative, the slope of $\zj(L_3)$ can be made arbitrarily
close to 0. The edge path distance between $-\frac{1}{2} + \zb$ and
$P_1$ is bounded since $a$ and $c$ are bounded, so
it follows that usually an arbitrarily long initial
portion of the directed segment of $\widehat{L}_3$ from $P_1$ to $P_2$
is in the boundary of either $\overline{Y}_m$ or
$\overline{Y}_{m-1}$. In particular, $P_1$ is usually in the boundary
of either $\overline{Y}_m$ or $\overline{Y}_{m-1}$.  Moreover since
$c\ge 2$, it follows that $m<0$, and so $P_1\notin \overline{Y}_0$.
Similarly, the half of $\widehat{L}_1$ with endpoint $P_1$ which does
not contain $\zb$ is disjoint from $\overline{Y}_0$.  On the other
hand, $\zb\in\overline{Y}_0$, and the half of $\widehat{L}_1$ with
endpoint $\zb$ which contains $1+\zb$ is even in $\overline{Y}_0$.
The previous two sentences together with the fact that $Q_1$ is the
image of $P_1$ under the rotation of order 2 about $\zb$ imply that
the half of $\widehat{L}_1$ with endpoint $Q_1$ which does not contain
$\zb$ is in $\overline{Y}_1$.  Thus $P_1$ is usually strictly on the
same side of $\zb$ as $-1+\zb$ in $\widehat{L}_1$, it is usually
not adjacent to $\zb$ in $S'$, and usually no tile
of $S'$ meets both $\widehat{L}_3$ between $P_1$ and $P_2$ and
$\widehat{L}_1$ between $Q_1$ and $R_1$.  This completes the proof of
statement 2 and half of statement 1 when $1\le c\le 10$.

Finally we prove statement 3 and the other half of statement 1 when
$1\le c\le 10$.  As in the previous paragraph, it is still true that
$-\frac{1}{2}+\zb$ is contained in $Y_m$, where $m$ is the ceiling of
$-\frac{c}{2}$.  It is also still usually true that $P_1$ is in
$\overline{Y}_n$ with either $n=m$ or $n=m-1$.  If $a\ge 3c$, then the
slope of $\zj(L_3)$ is negative and the directed segment of
$\widehat{L}_3$ from $P_1$ to $P_2$ is also in $\overline{Y}_n$.  If
$a<3c$, then this slope might be negative, but usually an arbitrarily
long initial portion of this directed segment is contained in
$\overline{Y}_n$.  Moreover, if $c\in \{1,2\}$, then $d\ge 0$, the
slope of $\zj(L_3)$ is nonpositive, and we may take $n=m$.  It follows
that $-c<n$ and that $-1+\zb\notin \overline{Y}_n$.  It is also true
that the half of $\widehat{L}_1$ with endpoint $P_1$ which contains
$\zb$ is contained in $\overline{Y}_n$.
The last two sentences combine to prove that $P_1$ is
usually strictly on the same side of $-1+\zb$ as $\zb$ in
$\widehat{L}_1$ and that $P_1$ is not adjacent to $-1+\zb$ in $S'$.
The translation $z\mapsto
z+2=z+2a\za^{-1}+2c \za^{-1}\zt$ stabilizes $\widehat{L}_1$ and it
takes the directed segment of $\widehat{L}_3$ from $P_1$ to $P_2$ to
the directed segment of $\widehat{L}_4$ from $R_1$ to $R_2$.  So
either the directed segment of $\widehat{L}_4$ from $R_1$ to $R_2$ is
contained in $\overline{Y}_{n+2c}$ or at least an arbitrarily long
initial portion of it is usually contained in
$\overline{Y}_{n+2c}$. Since $-1+b\notin \overline{Y}_n$, it follows
that $1+\zb\notin \overline{Y}_{n+2c}$.  Since the segment of
$\widehat{L}_1$ between $P_1$ and $Q_1$ is contained in the image of
$\overline{Y}_{n+2c}$ under the rotation of order 2 about $1+\zb$, it
follows that usually no tile of $S'$ meets both the segment of
$\widehat{L}_4$ between $R_1$ and $R_2$ and the segment of
$\widehat{L}_1$ between $P_1$ and $Q_1$.

This proves Lemma~\ref{lemma:pone}.

\end{proof}

So $P_1$ is usually strictly between $-1+\zb$ and $\zb$ in
$\widehat{L}_1$.  The next claim in the outline is that $P_2$ is
usually strictly between $-1+\zt+\zb$ and $\zt+\zb$ in
$\widehat{L}_2$.  To see this, note that the previous claim and
symmetry show that $R_2$ is usually strictly between $1+\zt+\zb$ and
$2+\zt+\zb$ in $\widehat{L}_2$.  This and a translation imply that
$P_2$ is usually strictly between $-1+\zt+\zb$ and $\zt+\zb$ in
$\widehat{L}_2$.

The next claim is that the closed segment of $\widehat{L}_3$ with
endpoints $P_1$ and $P_2$ is usually disjoint from its image under the
translation given by $z\mapsto z+R_1-P_1$.  This follows from
statement 1 of Lemma~\ref{lemma:approximations},
Lemma~\ref{lemma:threesixone}, and statement 2 of
Lemma~\ref{lemma:threesixtwo}.

In the outline of the proof of Theorem~\ref{thm:onetilea} a
fundamental domain $F$ for the orbifold fundamental group $\zG$ is
constructed.  The points $P_1$, $Q_1$, $R_1$, $P_2$, $Q_2$, $R_2$,
$\zb$, $1+\zb$, $\zt+\zb$, $1+\zt+\zb$ on the boundary of $F$
decompose the boundary of $F$ into ten edges, and there is a pairing
of these edges.  There is a corresponding decomposition of the
boundary of every tile of $S'$ into ten paired edges.  This and the
tiling of $F$ by tiles of $S'$ determine a finite subdivision rule
$\cR$.  The final claim is that the mesh of $\cR$ usually approaches 0
combinatorially.  To say that the mesh of $\cR$ approaches 0
combinatorially means that there exists a positive integer $n$ such
that the following two conditions hold, where $t$ is the tile type of
$\cR$.  Condition 1 is that every edge of $t$ properly subdivides in
$\cR^n(t)$.  Condition 2 is that no tile of $\cR^n(t)$ has two edges
contained in disjoint edges of $t$.  The next two lemmas verify that
$\cR$ usually satisfies conditions 1 and 2, and so the mesh of $\cR$
usually approaches 0 combinatorially.  The strategy is to first show
in Lemma~\ref{lemma:cndone} that if $\cR$ satisfies condition 2, then
it satisfies condition 1.  Lemma~\ref{lemma:cndtwo} completes the
verification by showing that $\cR$ satisfies condition 2.

\begin{lemma}\label{lemma:cndone} If $\cR$ satisfies condition 2, then
it satisfies condition 1.
\end{lemma}
  \begin{proof} The translates of the fundamental domain $F$ under the
universal orbifold group tile the plane.  Furthermore the points
$P_1$, $Q_1$, $R_1$, $P_2$, $Q_2$ $R_2$ are vertices of this tiling
with valence 3.  So if $v\in \{P_1,Q_1,R_1,P_2,Q_2,R_2\}$, then there
exists exactly one tile $s$ of $S'$ such that $v\in s\subseteq F$.

If $e$ is an edge of $F$ which does not subdivide, then $e$ is in fact
an edge of $S'$.  Moreover the previous paragraph implies that if $s$
is the tile of $S'$ with $e\subseteq s\subseteq F$, then two edges of
$s$ are contained in disjoint edges of $F$.  Therefore if condition 1
fails, then so does condition 2.

This proves Lemma~\ref{lemma:cndone}.

\end{proof}

\begin{lemma}\label{lemma:cndtwo} The subdivision rule $\cR$ usually
satisfies condition 2.
\end{lemma}
  \begin{proof} To prove this, we interpret condition 2 in terms of a
directed graph $G$ which was considered previously in \cite{CFPP}.
The vertices of $G$ are ordered pairs $(e_1,e_2)$, where $e_1$ and
$e_2$ are disjoint edges of the tile type $t$ of $\cR$.  There exists
a directed edge from the vertex $(e_1,e_2)$ to the vertex $(e_3,e_4)$
if and only if $\cR(t)$ contains a tile $s$ such that an edge of $s$
with edge type $e_3$ is contained in $e_1$ and an edge of $s$ with
edge type $e_4$ is contained in $e_2$.  Then $\cR$ satisfies condition
2 if and only if $G$ contains no directed cycles.  We will show that
$G$ usually has no directed cycles.

To facilitate the discussion we number the edge types of $\cR$ so that
the types of the ten edges of $F$ are numbered in counterclockwise
order beginning with the two edges of $F$ which contain $\zb$.  So the
two edges containing $\zb$ have types 1 and 2, the two edges
containing $1+\zb$ have types 3 and 4, the two edges containing
$1+\zt+\zb$ have types 6 and 7, and the two edges containing $\zt+\zb$
have types 8 and 9.

Let $e_1$ and $e_2$ be
distinct edges of $t$.  Let $s$ be any tile of $S'$, and let $E_1$ and
$E_2$ be the edges of $s$ corresponding to $e_1$ and $e_2$.  Then each
of $E_1$ and $E_2$ might be an edge of $S$ or half of an edge of $S$.
We say that $e_1$ and $e_2$ are $S$-\textbf{adjacent} if the edges of
$S$ containing $E_1$ and $E_2$ have a vertex in common.

Let $e_1$ and $e_2$ be disjoint edges of $t$ which are $S$-adjacent.  We
next show that $(e_1,e_2)$ is usually not the terminal vertex of a directed
edge of $G$.  For this we assume that $s$ is a tile of $S'$ contained
in $F$ such that the edges $E_1$ and $E_2$ of $s$ corresponding to
$e_1$ and $e_2$ are contained in edges of $F$.
Statement 1 of Lemma~\ref{lemma:pone} shows
that $P_1$ is usually not adjacent to either $-1+\zb$ or $\zb$ in $S'$.  Using
this and symmetry, we see that usually no edge of $F$ is a half edge of $S$,
that is, an edge of $S'$ which is not an edge of $S$.  Moreover, every
edge of $F$ has an endpoint which is a vertex of $S$.  It follows that
the edges of $F$ which contain $E_1$ and $E_2$ usually have a vertex in
common.
So these edges of $F$ are usually not disjoint.
Thus $(e_1,e_2)$ is usually not the
terminal vertex of a directed edge of $G$.  Therefore if $e_1$ and
$e_2$ are $S$-adjacent, then $(e_1,e_2)$ is usually not contained in a
directed cycle of $G$.

Next let $e_1$ and $e_2$ be edges of $t$ such that one of them has
edge type either 1, 2, 3, or 4 and one of them has edge type either 6,
7, 8, or 9.  The corresponding edges of $F$ are contained in
$\widehat{L}_1$ and $\widehat{L}_2$.  In this case
Lemma~\ref{lemma:onetwo} implies that $(e_1,e_2)$ is usually not the
initial vertex of a directed edge of $G$.  Therefore $(e_1,e_2)$ is
usually not contained in a directed cycle of $G$.

Now we prove that $G$ usually does not contain a directed cycle.  This
will be done by contradiction.  So suppose that $(e_1,e_2)$ is the
initial vertex and that $(e_3,e_4)$ is the terminal vertex of an edge
in a directed cycle of $G$.  The previous two paragraphs show that
$e_3$ and $e_4$ are usually not $S$-adjacent and usually one of them has type
either 5 or 10.

Suppose that $e_1$ has type 10.  The above discussion of $S$-adjacent
edges implies that the type of $e_2$ is usually not 1, 2, 8, or 9.  Statement 2
of Lemma~\ref{lemma:pone} implies that the type of $e_2$ is usually
not 3 or 4.  Statement 3 of Lemma~\ref{lemma:pone} and symmetry imply
that the type of $e_2$ is usually not 6 or 7.  So the type of $e_2$ is
usually 5.  Hence if $e_1$ corresponds to the edge of $F$ in
$\widehat{L}_3$, then $e_2$ usually corresponds to the edge of $F$ in
$\widehat{L}_4$.  Similarly, if $e_1$ corresponds to the edge of $F$
in $\widehat{L}_4$, then $e_2$ usually corresponds to the edge of $F$
in $\widehat{L}_3$.  The same is true for $e_3$ and $e_4$.  So we may
assume that one of $e_1$ and $e_2$ corresponds to the edge of $F$ in
$\widehat{L}_3$, the other corresponds to the edge of $F$ in
$\widehat{L}_4$, and $e_3$ and $e_4$ also correspond to two edges of
$S'$ which are dual to edges of $S^*$ in the direction of $\za^{-1}$.
Now Lemma~\ref{lemma:threesixone} shows that this is usually
impossible.

Thus $\cR$ usually satisfies condition 2.  This proves
Lemma~\ref{lemma:cndtwo}.

\end{proof}

This completes the proof that the mesh of $\cR$ usually approaches 0
combinatorially.  The proof of Theorem~\ref{thm:onetilea} is now
complete except for the cases in which $\za\in \{\zt,\zt+1,\zt+2\}$.
So suppose that $\za\in \{\zt,\zt+1,\zt+2\}$.  In the usual notation,
$c=1$ and $a\in \{0,1,2\}$.  Lemma~\ref{lemma:mx} implies that $d\in
\{a,a+1\}$.  As in the proof of Lemma~\ref{lemma:pone}, we may assume
that $b$ is arbitrarily negative.

We proceed by indicating the form of a suitable fundamental domain
$F$.  Rather than working directly with $F$, we find it easier to work
with $\zj(F)$.  By construction, the points of the lattice
$\zj(\za^{-1}\zL)$ are the midpoints of the edges of $T$ which are not
vertical.  So there are essentially two possibilities for $\zj(\zb)$;
either $\zj(\zb)$ is the midpoint of an edge of $T$ with positive
slope or $\zj(\zb)$ is the midpoint of an edge of $T$ with negative
slope.  The same holds for $\zj(\zt+\zb)$.  This leads to two
possibilities for the form of the edges of $\zj(F)$ with edge types 1,
2, 3, 4 and two possibilities for the form of the edges of $\zj(F)$
with edge types 6, 7, 8, 9.  In the following paragraphs we indicate
the form of $\zj(F)$ by giving the forms of the edges of $\zj(F)$ with
these edge types.  There are in general many ways to complete the
construction of $\zj(F)$.  It is a straightforward matter to do so.
One then checks that this defines a fundamental domain $F$ for the
orbifold fundamental group $\zG$.  As in the outline of the proof of
Theorem~\ref{thm:onetilea}, the fundamental domain $F$ has an edge
pairing which is combinatorially equivalent to the original edge
pairing on the tiles of $S'$.  As before we obtain a finite
subdivision rule $\cR$.  Finally one checks that the mesh of $\cR$
approaches 0 combinatorially.  Lemma~\ref{lemma:cndone} still holds,
and so to check that the mesh of $\cR$ approaches 0 combinatorially,
it suffices to check condition 2.  We leave the details to the reader.

Suppose that $(a,c)=(2,1)$.  Figure~\ref{fig:twoonea} shows four edge
paths.  We view these as building blocks for constructing fundamental
domains.  The first edge path in Figure~\ref{fig:twoonea} gives the
form of the edges of $\zj(F)$ with edge types 6, 7, 8,  and 9 when the
slope of the edge of $T$ containing $\zj(\zt+\zb)$ is positive.  The
second edge path occurs when this slope is negative.  The third edge
path gives the form of the edge of $\zj(F)$ with edge types 1, 2, 3,
and 4 when the slope of the edge of $T$ containing $\zj(\zb)$ is negative.
The fourth edge path occurs when this slope is positive.  The points
$\zj(\zb)$, $\zj(1+\zb)$, $\zj(\zt+\zb)$, and $\zj(1+\zt+\zb)$ are
indicated by small circles.  Circles indicating $\zj(\zb)$ are larger
than the others.  As in Figure~\ref{fig:newfundom}, the fundamental
domain $F$ contains points $P_1$, $Q_1$, $R_1$, $P_2$, $Q_2$, and $R_2$
and their images under $\zj$ are indicated by large dots.

For example, suppose that $(a,b,c,d)=(2,-3,1,2)$ and that $\zb=0$.
The slope of the edge of $T$ containing $\zj(\zb)=0$ is negative, so
we choose the third edge path in Figure~\ref{fig:twoonea}.  Since
$\zj(\zt+\zb)=\zj(\zt)=\zj(b \za^{-1}+d
\za^{-1}\zt)=\zj(-3\za^{-1}+2\za^{-1}\zt)$, the slope of the edge of
$T$ containing $\zj(\zt+\zb)$ is positive, so we also choose the first
edge path in Figure~\ref{fig:twoonea}.  This case is so simple that
these two choices determine $\zj(F)$, and we obtain the fundamental
domain shown in the left portion of Figure~\ref{fig:twooneb}.
The subdivision of the tile type of the corresponding finite
subdivision rule is shown in Figure~\ref{fig:fracthx2}.  The
tile type is the Gosper snowflake.

For another
example we take $(a,b,c,d)=(2,-6,1,2)$ with $\zb=0$.  In this example
the slope of the edge of $T$ containing $\zj(\zb)=0$ is still
negative.  The slope of the edge of $T$ containing
$\zj(\zt+\zb)=\zj(\zt)=\zj(b \za^{-1}+d
\za^{-1}\zt)=\zj(-6\za^{-1}+2\za^{-1}\zt)$ is also negative, so we
also choose the second edge path in Figure~\ref{fig:twoonea}.  The
edge paths in Figure~\ref{fig:twoonea} determine 8 of the 10 edges of
$\zj(F)$.  The two remaining edges must be translates of one another,
but they are not uniquely determined.  One choice is given in the
right portion of Figure~\ref{fig:twooneb}.  It is easy to check that
the meshes of these finite subdivision rules approach 0
combinatorially.  The situation is similar whenever $(a,c)=(2,1)$ and
the degree of $f$ is sufficiently large.

\begin{figure}
\centerline{\includegraphics{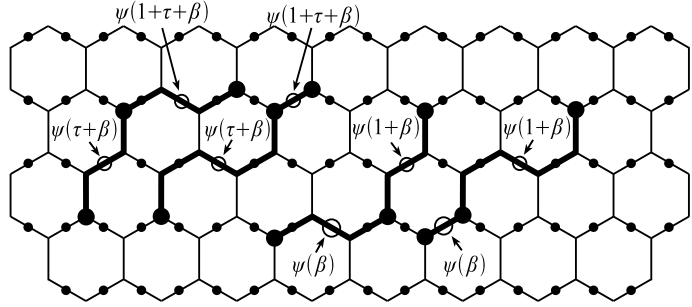}} \caption{Building
blocks for constructing fundamental domains for the cases in which
$(a,c)=(2,1)$.}
\label{fig:twoonea}
  \end{figure}

  \begin{figure}
\centerline{\includegraphics{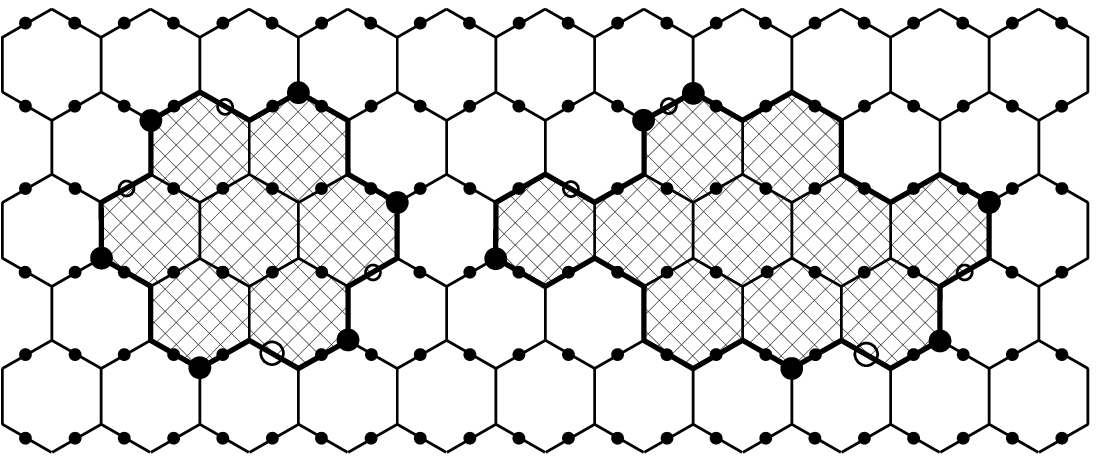}} \caption{Fundamental
domains for the cases in which $(a,b,c,d)=(2,-3,1,2)$ with $\zb=0$ and
$(a,b,c,d)=(2,-6,1,2)$ with $\zb=0$.}
\label{fig:twooneb}
  \end{figure}

\begin{figure}
\centerline{\includegraphics{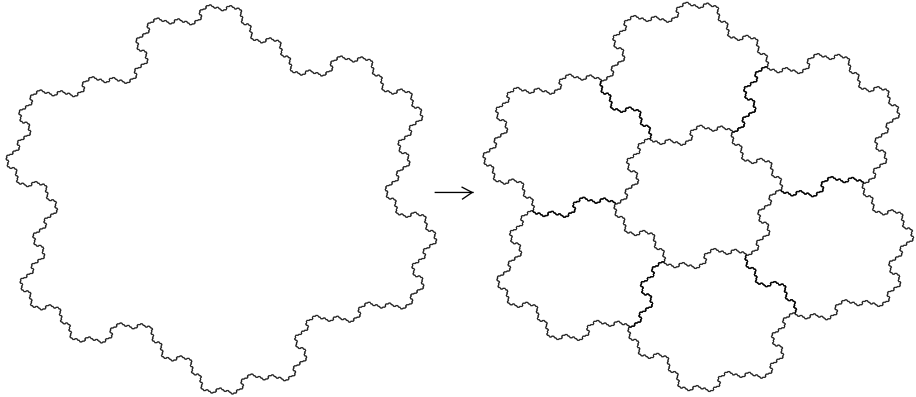}} \caption{The
subdivision of the tile type of the finite subdivision rule associated
to the Latt\`es map with $(a,b,c,d)=(2,-3,1,2)$.}
\label{fig:fracthx2}
\end{figure}

Next suppose that $(a,c)=(1,1)$.  Figure~\ref{fig:oneonea} shows
building blocks for this case.  For example, suppose that
$(a,b,c,d)=(1,-3,1,2)$ and $\zb=0$.  The slope of the edge of $T$
containing $\zj(\zb)=0$ is negative, so we choose the fourth edge path
in Figure~\ref{fig:oneonea}.  Since $\zj(\zt+\zb)=\zj(\zt)=\zj(b
\za^{-1}+d \za^{-1}\zt)=\zj(-3\za^{-1}+2\za^{-1}\zt)$, the slope of
the edge of $T$ containing $\zj(\zt+\zb)$ is positive, so we also
choose the second edge path in Figure~\ref{fig:oneonea}.  We complete
these two edge paths in the straightforward way to the fundamental
domain shown in the left portion of Figure~\ref{fig:oneoneb}.  For
another example we take $(a,b,c,d)=(1,-6,1,2)$ and $\zb=0$.  Since
$\zj(\zb)=0$, the slope of the edge of $T$ containing $\zj(\zb)$ is
negative, so we choose the fourth edge path in
Figure~\ref{fig:oneonea}.  Since
$\zj(\zt+\zb)=\zj(\zt)=\zj(-6\za^{-1}+2\za^{-1}\zt)$, the slope of the
edge of $T$ containing $\zj(\zt+\zb)$ is negative, so we also choose
the first edge path in Figure~\ref{fig:oneonea}.  One way to complete
these two edge paths to a fundamental domain is shown in the right
portion of Figure~\ref{fig:oneoneb}.  Just as for the case in which
$(a,c)=(2,1)$, whenever $(a,c)=(1,1)$ and the degree of $f$ is
sufficiently large we obtain a fundamental domain which determines a
finite subdivision rule whose mesh approaches 0 combinatorially.

  \begin{figure}
\centerline{\includegraphics{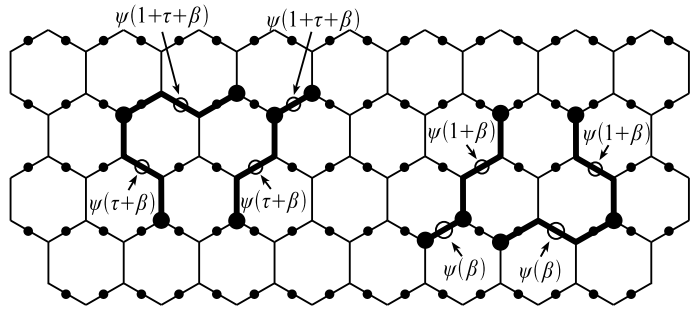}} \caption{Building blocks
for constructing fundamental domains for the cases in which
$(a,c)=(1,1)$.}
\label{fig:oneonea}
  \end{figure}

  \begin{figure}
\centerline{\includegraphics{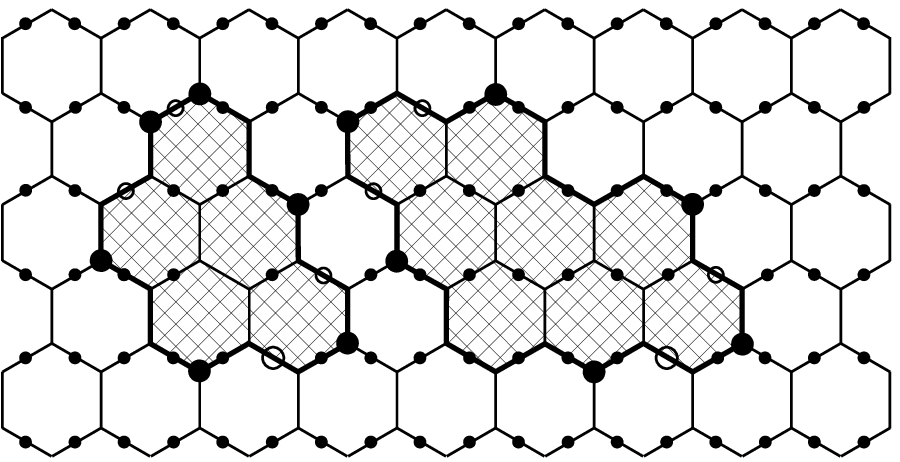}} \caption{Fundamental
domains for the cases in which $(a,b,c,d)=(1,-3,1,2)$ with $\zb=0$ and
$(a,b,c,d)=(1,-6,1,2)$ with $\zb=0$.}
\label{fig:oneoneb}
  \end{figure}

Finally suppose that $(a,c)=(0,1)$.  Figure~\ref{fig:zeroonea} shows
building blocks for this case.  For example, suppose that
$(a,b,c,d)=(0,-4,1,1)$ and $\zb=0$.  The slope of the edge containing
$\zj(\zb)=0$ is negative, so we choose the third edge path in
Figure~\ref{fig:zeroonea}.  Since $\zj(\zt+\zb)=\zj(\zt)=\zj(b
\za^{-1}+d \za^{-1}\zt)=\zj(-4\za^{-1}+\za^{-1}\zt)$, the slope of the
edge of $T$ containing $\zj(\zt+\zb)$ is positive, so we also choose
the second edge path in Figure~\ref{fig:zeroonea}.  We complete
these two edge paths as shown in the left portion of
Figure~\ref{fig:zerooneb} to obtain a fundamental domain.  For another
example we take $(a,b,c,d)=(0,-7,1,0)$ and $\zb=1$.  The slope of the
edge containing $\zj(\zb)=\zj(1)=\zj(a \za^{-1}+c
\za^{-1}\zt)=\zj(\za^{-1}\zt)$ is positive, so we choose the fourth
edge path in Figure~\ref{fig:zeroonea}.  Since
$\zj(\zt+\zb)=\zj(\zt+1)=\zj((a+b) \za^{-1}+(c+d)
\za^{-1}\zt)=\zj(-7\za^{-1}+\za^{-1}\zt)$, the slope of the edge of
$T$ containing $\zj(\zt+\zb)$ is negative, so we choose the first edge
path in Figure~\ref{fig:zeroonea}.  We complete these two edge paths
as shown in the right portion of Figure~\ref{fig:zerooneb} to obtain a
fundamental domain.  As before whenever $(a,c)=(0,1)$ and the degree of
$f$ is sufficiently large we obtain a fundamental domain which
determines a finite subdivision rule whose mesh approaches 0
combinatorially.

  \begin{figure}
\centerline{\includegraphics{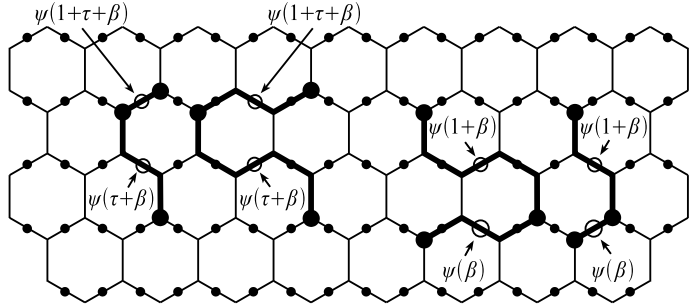}} \caption{Building blocks
for constructing fundamental domains for the cases in which
$(a,c)=(0,1)$.}
\label{fig:zeroonea}
  \end{figure}

  \begin{figure}
\centerline{\includegraphics{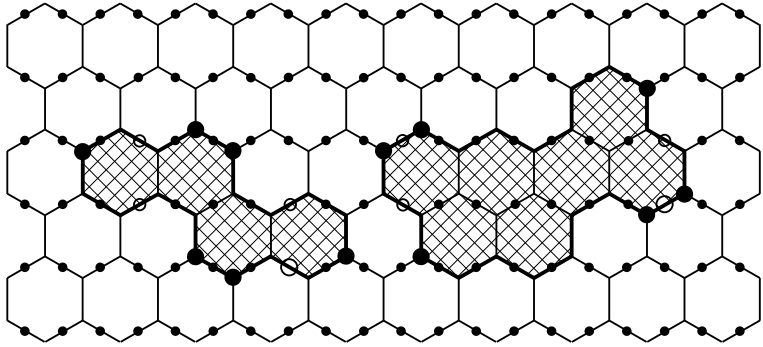}} \caption{Fundamental
domains for the cases in which $(a,b,c,d)=(0,-4,1,1)$ with $\zb=0$ and
$(a,b,c,d)=(0,-7,1,0)$ with $\zb=1$.}
\label{fig:zerooneb}
  \end{figure}

This completes the proof of Theorem~\ref{thm:onetilea}.

\end{proof}

\begin{thm}\label{thm:onetileb} Every nonrigid Latt\`{e}s map is the
subdivision map of a finite subdivision rule with one tile type, and
the tile has the form of type 1 in Figure~\ref{fig:types}.
\end{thm}
  \begin{proof}Let $f$ be a nonrigid Latt\`{e}s map.  Let
$\widetilde{f}(z)=\za z+\zb$ be a lift of $f$.  Then $\za\in \bZ$, and
we may assume that $\za\ge 2$.  Let 1 and $\zt$ form a basis of the
usual lattice $\zL$ with $\text{Im}(\zt)>0$.  We may add any element
of $2\zL$ to $\zb$, so we assume that $\zb\in \{0,1,\zt,1+\zt\}$.

Recall that Figure~\ref{fig:types} shows five types of planar tilings.
We construct a tiling $S'$ of the plane with type 1 so that every tile
of $S'$ is a parallelogram.  Furthermore one of the parallelograms of
$S'$ is chosen to be as in Figure~\ref{fig:funprlg}.  If $\zb=1$, then
we choose the leftmost parallelogram.  If $\zb=\zt$, then we choose
the middle parallelogram.  If $\zb=1+\zt$, then we choose the
rightmost parallelogram.  If $\zb=0$, then we choose any of these.
Let $P$ denote the chosen parallelogram.

  \begin{figure}
\centerline{\includegraphics{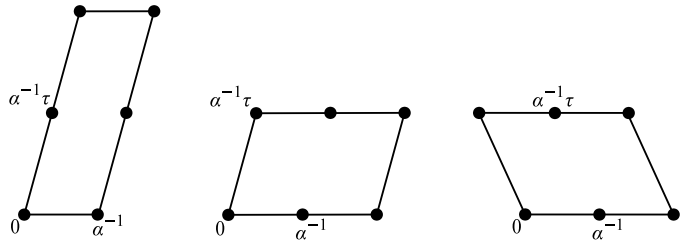}} \caption{The parallelogram
$P$.}
\label{fig:funprlg}
  \end{figure}

The tiling $S'$ is constructed so that $\zb$ is a vertex with valence
4.  From this it follows that $F=\widetilde{f}(P)$ is a union of tiles
of $S'$.  The parallelogram $P$ is really a hexagon because it has two
vertices with valence 2.  The two edges which contain one of the
vertices with valence 2 are interchanged by a rotation of order 2.
This pairs four of the six edges of $P$.  The remaining two edges of
$P$ are paired by a translation.  In the same way $F$ has six edges
which are paired so that $\widetilde{f}$ respects these edge pairings.
As in the proof of Theorem~\ref{thm:onetilea}, we obtain a finite
subdivision rule $\cR$ with one tile type.  Not as in the proof of
Theorem~\ref{thm:onetilea}, in the present situation $f$ is already
the subdivision map of $\cR$.

This proves Theorem~\ref{thm:onetileb}.

\end{proof}

\begin{cor}\label{cor:main} Every map in almost every analytic
conjugacy class of Latt\`{e}s maps is the subdivision map of a finite
subdivision rule with one tile type.
\end{cor}
  \begin{proof} This follows from Theorem~\ref{thm:onetilea},
Theorem~\ref{thm:onetileb}, and the last statement in
Section~\ref{sec:defn}, namely, that there are only finitely many
analytic conjugacy classes of rigid Latt\`{e}s maps with bounded
degree.

\end{proof}

\section{An exceptional quadratic Latt\`{e}s map}\label{sec:badquadratic}\nosubsections

This section deals with the Latt\`{e}s map $f$ with lift
$\widetilde{f}(z)=\za z$, where $\za=(1+\sqrt{-7})/2$ and the usual
lattice $\zL$ is generated by 1 and $\za$.  We prove that $f$ is not
the subdivision map of a finite subdivision rule with one tile type,
$f$ is not the subdivision map of a finite subdivision rule with two
tile types and 1-skeleton of the subdivision complex
homeomorphic to a circle, and $f$ is the
subdivision map of a finite subdivision rule with two tile types.

The minimal quadratic equation satisfied by $\za$ is $\za^2-\za+2=0$.
Since the constant term is $\za \overline{\za}=2$, the map $f$ is
indeed quadratic.  Moreover $\widetilde{f}(0)=0$,
$\widetilde{f}(1)=\za$, $\widetilde{f}(\za)=\za^2\equiv \za\mod 2\zL$
and $\widetilde{f}(\za+1)=\za^2+\za\equiv 0\mod 2\zL$.  Let $A=\wp
(\za+1)$, $B=\wp (1)$, $X=\wp (0)$, and $Y=\wp (\za)$.  Then $f(A)=X$,
$f(B)=Y$, $f(X)=X$, and $f(Y)=Y$.  Since $A$, $B$, $X$, and $Y$ are
the only postcritical points of $f$,
one critical point of $f$ maps to $A$ and one critical point of $f$
maps to $B$.  We normalize so that $\infty$ and 0 are the critical
points of $f$ with $f(\infty)=A$ and $f(0)=B$.
Figure~\ref{fig:dirdgraph} shows the mapping scheme for $f$.  Integers
next to arrows are local mapping degrees.

We can normalize $f$ so that $f(z) = \frac{Az^2 + B}{z^2 + 1}$.
Since $X = f(X) = f(A)$, $Y = f(Y) = f(B)$, and $f$ is even, $X =
-A$ and $Y = -B$.  Hence $\frac{A^3 + B}{A^2 + 1} = -A$ and
$\frac{AB^2 + B}{B^2 + 1} = -B$.  This gives $B = -2A^3 - A$ and
$AB + 1 = - B^2 - 1$. Substituting $B = - A(2A^2 + 1)$ into the
equation $AB + 1 = - B^2 -1$ yields $0 = 2A^6  + A^4 + 1 =
(A^2+1)(2A^4 - A^2 + 1)$. If $A = \pm i$ then $B = A$, so $2A^4 -
A^2 + 1 = 0$. Hence $f(z) = \frac{Az^2 + B}{z^2 + 1}$, where $A^2
= \frac{1 \pm \sqrt{-7}}{4}$ and $B = -A(2A^2 + 1)$. That is,
either $2A^2 = \alpha$ or $2A^2 = \overline{\alpha}$. (By choosing
$0$ and $\infty$ to be the fixed postcritical points, in
\cite[Section B.4]{M} Milnor gives the alternate form $f(z) = z
\frac{z + \alpha^2}{\alpha^2 z + 1}$ for $f$.)

\begin{figure}
\centerline{\includegraphics{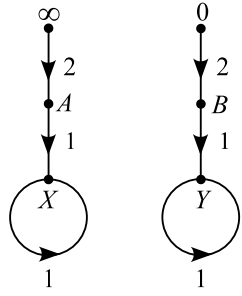}} \caption{The mapping
scheme of $f$.}
\label{fig:dirdgraph}
\end{figure}

We first prove by contradiction that $f$ is not the subdivision map of
a finite subdivision rule with one tile type.  Suppose that $f$ is the
subdivision map of a finite subdivision rule $\cR$ with one tile type.
Let $T$ be the 1-skeleton of the subdivision complex of $\cR$. By
Lemma~\ref{lemma:prune}, we may and do assume that the only vertices
of $T$ with valence $1$ or $2$ are postcritical points.

\begin{lemma}\label{lemma:vxcrit}  No vertex of $T$ is a
critical point of $f$.
\end{lemma}
  \begin{proof} If a critical point of $f$ is a vertex of $T$, then it
is an accidental vertex (not a postcritical point).  A tree with type
1 or 2 has no accidental vertices.  The postcritical vertices of trees
of type 3 and 5 all have valence 1, and so a critical vertex would
have valence at most 2, which is impossible.  Suppose that $T$ has
type 4.  Then $T$ has just one accidental vertex, the vertex with
valence 3.  By symmetry we may assume that this vertex is $\infty$.
As before since $\infty$ has valence greater than 2, $f(\infty)=A$
must have valence greater than 1.  So $A$ is the vertex with valence
2.  But then $f(A)$ also has valence at least 2 because the degree of
$f$ at $A$ is 1.  This is impossible.  This shows that no vertex of
$T$ is a critical point of $f$.

\end{proof}

\begin{lemma}\label{lemma:tescrit}  Let $p$ and $q$ be
distinct vertices of $T$ such that $f(p),f(q)\in \{p,q\}$.  Then the
open segment $(p,q)$ of $T$ contains a critical point.
\end{lemma}
  \begin{proof} If $f$ maps both $p$ and $q$ to the same vertex, then
because $f$ maps the segment $(p,q)$ into the tree $T$, the image of
$(p,q)$ under $f$ must be folded somewhere.  Such a fold is the image
of a critical point.  So $(p,q)$ contains a critical point in this
case.

Now suppose that $f(p)=p$ and $f(q)=q$ and that $(p,q)$ contains no
critical point of $f$.  Then $f$ does not fold the segment $[p,q]$
anywhere.  It follows that $f$ maps $[p,q]$ homeomorphically onto
$[p,q]$.  But then the edges in $[p,q]$ never subdivide.  This implies
that the mesh of our finite subdivision rule does not go to 0,
contrary to the fact that $f$ is an expanding map.  So the lemma holds
if $f(p)=p$ and $f(q)=q$.

Finally, the case in which $f(p)=q$ and $f(q)=p$ can be proved just as
in the case in which $f(p)=p$ and $f(q)=q$.

This proves Lemma~\ref{lemma:tescrit}.

\end{proof}

\begin{lemma}\label{lemma:neiab}  Neither $A$ nor $B$ is in the
closed segment $[X,Y]$.
\end{lemma}
  \begin{proof} Proceeding by contradiction, we may assume by symmetry
that $A\in [X,Y]$.  Since $f(A)=X$ and $f(Y)=Y$ there exists a point
in $(A,Y)$ which maps to $A$.  This point must be $\infty$.  Since $f$
maps $A$ to $X$ with degree 1, $f([X,Y])$ contains a nontrivial
segment which meets $[X,Y]$ at $X$.  But then there exists a
nontrivial segment $s$ in $T$ which contains a vertex of $T$ with
valence 1 and has $s \cap [X,Y] = X$.  This vertex with valence 1 must
be $B$.  Since $f(X)=X$ and $f(B)=Y$, it follows that $s$ contains a
preimage of $A$, namely, $\infty\in s$.  This is impossible.

\end{proof}

Now we apply Lemma~\ref{lemma:tescrit} to see that there exists a
critical point in $[X,Y]$.  By symmetry we may assume that $\infty\in
[X,Y]$.  Lemma~\ref{lemma:neiab} shows that $A\notin [X,Y]$.  Let $v$
be the vertex of $[X,Y]$ which is closest to $A$ in $T$.  If $v=Y$,
then because $f(X)=X$ and $f(\infty)=A$, $f$ maps some point in
$(X,\infty)\subseteq (X,Y)$ to $Y$.  This point must be $B$, which is
impossible by Lemma~\ref{lemma:neiab}.
Hence $v\ne Y$.  Likewise $v\ne X$.  So $v$ is
a vertex in $(X,Y)$ with valence at least 3.

In this paragraph we assume that $v$ is the only vertex of $T$ with
valence at least 3.  As in the previous paragraph we see that
$v$ is the vertex in $[X,Y]$ which is closest to both $A$ and $B$.
Lemma~\ref{lemma:vxcrit} implies that $v$ is not
a critical point.  Hence $f(v)$ has valence at least 3, and so
$f(v)=v$.  Since $f(X)=X$ and $f(v)=v$, Lemma~\ref{lemma:tescrit}
implies that $(X,v)$ contains a critical point $p$.  Since the
critical point maps to $A$ or $B$, $(X,v)$ contains a preimage of $v$.
Likewise $(v,Y)$ contains a critical point $q$, and
$(q,Y)$ contains a preimage of $v$.  We now have three points mapping
to $v$, which is impossible.

To complete this argument we assume that $T$ contains two vertices $u$
and $v$ with valence 3, that one of them, $v$, is in $(X,Y)$, and that
either $u\in (X,Y)$ or $v$ is the vertex in $[X,Y]$ closest to $u$.
As in the previous paragraph, we see that $f(u),f(v)\in \{u,v\}$.
First suppose that $f(u)=u$ and $f(v)=v$.  Lemma~\ref{lemma:tescrit}
implies that $(u,v)$ contains a critical point.  So does the segment
from $X$ to $[u,v]$ as well as the segment from $Y$ to $[u,v]$.  This
is impossible.  Next suppose that $f(u)=v$ and $f(v)=u$.  Let $w\in
(u,v)$.  Then $(u,v)$ contains a preimage of $w$.  So does the segment
from $X$ to $[u,v]$ as well as the segment from $Y$ to $[u,v]$.  We
now have three preimages of $w$, which is impossible.  We are left
with the case in which $f(u)=f(v)\in \{u,v\}$. Suppose in addition
that $u\notin [X,Y]$.  Recall that $\infty$ is a critical point in
$[X,Y]$.  Choose $z\in \{X,Y\}$ so that $v\notin [z,\infty]$.  Then
the image of $[z,\infty]$ covers $[u,v]$, and so $[z,\infty]$ contains
a preimage of $f(u)=f(v)$, giving three preimages, which is
impossible.  So $u\in [X,Y]$.  Now choose $z\in \{X,Y\}$ so that
neither $u$ nor $v$ lies in $(z,f(v))$.  Lemma~\ref{lemma:tescrit}
implies that $(z,f(v))$ contains a critical point $p$.  But then
$(z,p)$ contains a preimage of $f(v)$, giving a third preimage of
$f(v)$, which is once more impossible.

The proof that $f$ is not given by a finite subdivision rule with one
tile type is now complete.

We next prove by contradiction that $f$ is not given by a finite
subdivision rule with two tile types with 1-skeleton
of the subdivision complex homeomorphic
to a circle.  Suppose that $f$ is given by a finite subdivision
rule with two tile types with 1-skeleton $G$
of the subdivision complex homeomorphic to a
circle. If some open edge of $G$ is not in $f(G)$, then we delete
that open edge to obtain a finite subdivision rule for $f$ with
one tile type.  Since such a finite subdivision rule does not
exist, it follows that $f(G)=G$.

In this paragraph we assume that the restriction of $f$ to $G$ is
locally injective.  Then the restriction of $f$ to $G$ is a covering
map with degree either 1 or 2.  If this degree is 2, then every point
in $G$ has two preimages, but both $A$ and $B$ have at most one
preimage in $G$.  If this degree is 1, then the restriction of $f$ to
$G$ is a homeomorphism.  In this case the edges of $G$ do not
subdivide and the mesh of our finite subdivision rule fails to go to
0.  So the restriction of $f$ to $G$ is not locally injective.

Since the restriction of $f$ to $G$ is not locally injective, $G$
contains a critical point.  We may thus assume that $f$ is not locally
injective at $\infty\in G$.  Let $e_1$ and $e_2$ be the edges of $G$
which contain $A$ and suppose that $e_1$ is doubly covered by $f$ near
$\infty$.  But then $e_2$ is not in $f(G)$, for if it is, then some
point other than $\infty$ maps to $A$, which is not the case.
Thus $f$ is not given by a finite subdivision rule
with two tile types with 1-skeleton homeomorphic to a circle.

We finally prove that $f$ is the subdivision map of a finite
subdivision rule with two tile types. The left portion of
Figure~\ref{fig:exone} shows a parallelogram which is a fundamental
domain $F$ for the orbifold fundamental group $\zG$ of $f$.  In
addition to the vertices at the corners of $F$, there are vertices at
1 and $\za$.  There is also an extra edge joining 0 and $\za$.  The
hatched region in the right portion of Figure~\ref{fig:exone} shows
$\widetilde{f}^{-1}(F)$.  To verify this it is useful to note that
$\za^2-\za+2=0$ implies that $\za(1-\za)=2$ and so $1-\za=2\za^{-1}$.
The last identity shows that $2\za^{-1}$ is drawn correctly.  Moreover
the fact that 0, 2 and $\za+1$ are vertices of the parallelogram $F$
implies that 0, $2\za^{-1}$ and $\za^{-1}+1$ are vertices of the
parallelogram $\widetilde{f}^{-1}(F)$.  Thus $\widetilde{f}^{-1}(F)$
is drawn correctly.  Let $T$ be the tiling of the plane by the images
of $F$ under $\zG$.  Figure~\ref{fig:extwo} shows a part of $T$ drawn
with dashed line segments and part of $\widetilde{f}^{-1}(T)$ drawn
with solid line segments.  Few dashed line segments of $T$ are visible
because most of them are obscured by the line segments of
$\widetilde{f}^{-1}(T)$.  Now it is clear that there exists a
$\zG$-equivariant isotopy from $T$ to $\widetilde{f}^{-1}(T)$ rel
$\zL$.  This determines a finite subdivision rule $\cR$ with two tile
types, one triangle and one pentagon.  Figure~\ref{fig:exthree} gives
a schematic description of the subdivisions of the tile types of
$\cR$.  The edges drawn with dashes in Figure~\ref{fig:exthree} are
the edges whose edge types correspond to the line segments drawn with
dashes in Figure~\ref{fig:extwo}.  One easily checks that $\cR$ has
bounded valence and mesh approaching 0 combinatorially.  It now
follows as in the proof of the main theorem of \cite{CFP2} that $f$ is
the subdivision map of a finite subdivision rule isomorphic to $\cR$.
A stereographic projection of the subdivision complex for $f$ is shown
in Figure~\ref{fig:exfour} with hats indicating images of 0, 1, $\za$,
and $\za+1$.

\begin{figure}
\centerline{\includegraphics{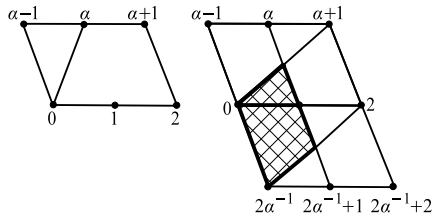}}
\caption{The fundamental fundamental domain $F$ and
$\widetilde{f}^{-1}(F)$.}
\label{fig:exone}
\end{figure}

\begin{figure}
\centerline{\includegraphics{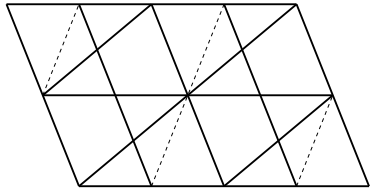}}
\caption{A small part of the tiling $T$ and $\widetilde{f}^{-1}(T)$.}
\label{fig:extwo}
\end{figure}

\begin{figure}
\centerline{\includegraphics{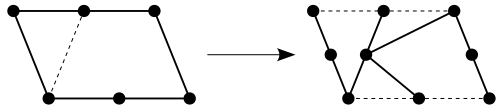}}
\caption{The finite subdivision rule $\cR$.}
\label{fig:exthree}
\end{figure}

\begin{figure}
\centerline{\includegraphics{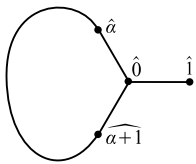}}
\caption{The subdivision complex for $f$.}
\label{fig:exfour}
\end{figure}

\end{document}